\documentclass[a4paper, 11pt]{article}

\usepackage[small]{titlesec}
\usepackage{tikz-cd}
\usepackage{booktabs}
\usepackage[T1]{fontenc}
\usepackage[utf8]{inputenc}
\usepackage[english]{babel}
\usepackage[a4paper,top=4cm,bottom=3cm,left=2.5cm,right=2.5cm]{geometry}
\usepackage{amsfonts}
\usepackage{mathtools}
\usepackage{amsmath}

\usepackage{lmodern}
\usepackage{empheq}
\usepackage{amsthm}
\usepackage{amssymb}
\usepackage{booktabs}
\usepackage{caption}
\usepackage{siunitx}
\usepackage{float}
\usepackage{dsfont}
\usepackage{amssymb}
\usepackage{graphicx}
\usepackage{stmaryrd}
\usepackage{color}
\usepackage{pifont}%
\usepackage{enumerate}
\usepackage{enumitem}
\setlist[description]{leftmargin=\parindent,labelindent=\parindent}
\usepackage{listings} 
\usepackage{amsmath}
\usepackage{amscd}
\usepackage{braket}
\usepackage{physics}
\usepackage{tikz}
\usepackage{calc}
\usepackage[hidelinks]{hyperref}

\newcommand{\dom}{\mathrm{dom}}

\newcommand{\R}{\mathbb{R}}
\newcommand{\sph}{\mathbb{S}}
\newcommand{\K}{\mathcal{K}}
\newcommand{\N}{\mathbb{N}}
\newcommand{\haus}{\mathcal{H}}
\newcommand{\Coco}{\Conv_{\mathrm{cd}}(\mathbb{R}^n)}
\newcommand{\Nu}{\boldsymbol{\nu}}

\newtheorem{theorem}{Theorem}[section]
\newtheorem*{theorem*}{Theorem}
\newtheorem{corollary}[theorem]{Corollary}

\newtheorem{lemma}[theorem]{Lemma}
\newtheorem{definition}[theorem]{Definition}

\theoremstyle{definition}
\newtheorem{remark}[theorem]{Remark}

\newcommand*\sq{\mathbin{\vcenter{\hbox{\rule{.3ex}{.3ex}}}}}

\usepackage{pgfplots}
\pgfplotsset{compat=1.17}

\newcommand{\Conv}{\mathrm{Conv}}

\newcommand{\Cosc}{\mathrm{Conv}_{\mathrm{coe}}}

\newcommand{\epi}{\mathrm{epi}}
\newcommand{\V}{\mathrm{Vol}}

\usepackage{tikz-cd} 

\newcommand\blfootnote[1]{%
	\begingroup
	\renewcommand\thefootnote{}\footnote{#1}%
	\addtocounter{footnote}{-1}%
	\endgroup
}
\title{First variation of functional Wulff shapes}
\author{Jacopo Ulivelli}
\date{}

\begin{document}
	\maketitle
	
	\begin{abstract}
		Using an approximation procedure we establish, for suitable perturbations, a weighted functional version of Aleksandrov's variational lemma in the family of convex functions with compact domain. The resulting formula is then applied to evaluate the first variation of a class of functionals on log-concave functions. In particular, we extend a recent result by Huang, Liu, Xi, and Zhao \cite{dual_curvature_log}.
	\end{abstract}
	
	\section{Introduction}
	
	This paper follows in the footsteps of the so-called functionalization of convex geometry. See, for example, \cite[Chapter 9]{Asymp_II}. We focus on the family of convex functions 
	\[\Conv(\R^n)\coloneq\{u:\R^n \to \R\cup \{+ \infty\} \colon u\text{ is convex, lower semi-continuous, } u \not \equiv +\infty\}.\] For $u \in \Conv(\R^n)$, consider the functional volume
	\begin{equation}\label{eq:logconc_vol}
		\mu_n(u)\coloneqq \int_{\R^n}e^{-u(x)}\, dx.
	\end{equation}
	This choice is justified by the Pr\'ekopa-Leindler inequality (see, for example, \cite{GARDBM}) and the functional Blaschke-Santal\'o inequality \cite{Func_Sant}, where $\mu_n$ replaces the volume in the corresponding classical inequalities for compact convex sets. A different angle on measures of convex functions has been pursued in the recently developing theory of valuations on convex functions \cite{cocofu, ColesantiEtAlHadwigertheoremconvex2020,CLM_hom, ColesantiEtAlHadwigertheoremconvex2021, Had3,Had4,KnoerrEquiv,Knoerrsupportduallyepi2021, Knoerr_singular,Knoerr_smooth, FromConvToFunc2023,Li_Legendre, Milman_AGA_perspectives, Mussnig_polar}.
	
	Consider the family $\K^n$ of non-empty, compact, and convex subsets of $\R^n$. We denote by $\V$ the $n$-dimensional Lebesgue measure on $\R^n$ and write $K+L$ for the Minkowski sum of for $K, L \in \K^n$. It is a classical fact (see, for example, \cite[(5.33)]{schneider_2013}) that the right derivative of $\V(K+tL)$ at $0$ takes the explicit form
	\begin{equation}\label{eq:intro_c}
		\left. \frac{d}{dt}\V(K+tL)\right|_{t=0^{+}}=\int_{\sph^{n-1}} h_L \, dS_K,
	\end{equation}
	where $\sph^{n-1}$ is the Euclidean unit sphere, $h_L$ is the support function of $L$, and $S_K$ is the surface area measure of $K$ (more on these definitions in Section 2). A similar question can be posed in $\Conv(\R^n)$. In this case, $\V$ is replaced by $\mu_n$, and Minkowski addition by infimal convolution, defined as \[u \square v(x)\coloneqq \inf_{y \in \R^n}\{u(y)+v(x-y)\}\] for $u,v \in\Conv(\R^n)$. Colesanti and Fragal\`a \cite{ColFrag} investigated the first variation
	\begin{equation}\label{eq:intro_fvar}
		\left. \frac{d}{dt}\mu_n(u\square (t\sq v))\right|_{t=0^{+}}
	\end{equation}
	for $u,v \in \Conv(\R^n)$ in a suitable class (see Section 2 for the definition of $t\sq v$). The same authors further explored the subject of functional mixed volumes together with Bobkov \cite{Quermass_Bob}. Later, Cordero-Erasquin and Klartag \cite{CorKlar} studied \eqref{eq:intro_fvar} in connection with moment measures in the family of essentially continuous convex functions. Rotem \cite[Theorem 1.5]{RotSur2} finally obtained a general formula for the first variation \eqref{eq:intro_fvar}, which reads as follows. For $u \in \Conv(\R^n)$ we write $\dom(u)\coloneqq \{x \in \R^n: u(x)<+\infty\}$ and we denote its gradient by $\nabla u$, and for a real valued function $g$ the function $g^*$ is its Fenchel-Legendre transform \[g^*(x)\coloneqq \sup_{y \in \R^n}\{x \cdot y -u(y)\}.\] For a convex set $K$, we denote its Gauss map by $N_K$.
	\begin{theorem}[Rotem]\label{thm:intro_rot}
		Let $u,v \in \Conv(\R^n)$. If $0<\mu_n(u)<+\infty$, then
		\begin{equation}\label{eq:intro_rot1}
			\begin{gathered}
				\left. \frac{d}{dt}\mu_n(u\square (t\sq v))\right|_{t=0^{+}}\\=\int_{\dom(u)}v^{*}(\nabla u(x))e^{-u(x)}\,dx+\int_{\partial \dom(u)}h_{\dom(v)}(N_{\dom(u)}(y))e^{-u(y)}\,d\mathcal{H}^{n-1}(y).
			\end{gathered}
		\end{equation}
	\end{theorem}
	Equivalently, \eqref{eq:intro_rot1} can be rewritten as 
	\begin{equation}\label{eq:intro_rot2}
		\left. \frac{d}{dt}\mu_n(u\square (t\sq v))\right|_{t=0^{+}}=\int_{\R^n}v^{*}(y)\, d\mathcal{M}_u(y)+\int_{\sph^{n-1}}h_{\dom(v)}(\Nu)\, d\mathcal{S}_u(\Nu),
	\end{equation}
	where $\mathcal{M}_u$ and $\mathcal{S}_u$ are the push-forward of the measures with densities $e^{-u(x)}dx$ and $e^{-u(y)}d\haus^{n-1}(y)$ obtained through the gradient of $u$ and the Gauss map of $\dom(u)$ respectively. This form has a striking parallelism with \eqref{eq:intro_c}, which is recovered if $u$ and $v$ are the indicator functions of $K, L \in \K^n$.

	More recently, Huang, Liu, Xi, and Zhao \cite{dual_curvature_log} investigated this problem in a generalized setting connected to the dual curvature measures of convex bodies introduced by Huang, Lutwak, Yang, and Zhang \cite{OPstuff}, enriching the variety of topics that this functional-geometric approach embraces. Also, see \cite{dual_curvature_orl}.
	
	For $K \in \K^n$, it is possible to consider more general perturbations than the one obtained via Minkowski addition, preserving the structure of \eqref{eq:intro_c}: These perturbations are obtained as Wulff shapes (also known as Aleksandrov bodies), which we present in detail in Section 3. The Wulff shape construction, which depends on a function $f \in C(\sph^{n-1})$, acts on the support function $h_K$, and for every $t\in \R$ generates a (possibly empty) compact convex set $F_t K$ whose support function is the maximal one with the property \[h_{F_t K}\leq h_K + tf.\] For $t \geq 0$ and $f=h_L$ for some $L \in \K^n$, $F_t K=K+tL$. A classical lemma due to Aleksandrov (see, for example, \cite[Lemma 7.5.3]{schneider_2013}) generalizes \eqref{eq:intro_c} to the formula
	\begin{equation}\label{eq:intro_Al}
		\left. \frac{d}{dt}\V(F_t K)\right|_{t=0}=\int_{\sph^{n-1}} f \, dS_K.
	\end{equation}
	A clear advantage of \eqref{eq:intro_Al} over \eqref{eq:intro_c} is that it provides a double-sided derivative
	This property makes \eqref{eq:intro_Al} a perfect tool to solve variational problems, such as the classical Minkowski problem (see, for example, \cite[Section 8]{schneider_2013}) and its many extensions. The original Minkowski problem can be formulated as follows: For a given Borel measure $\mathcal{S}$ on $\sph^{n-1}$, find sufficient and necessary conditions so that there exists $K \in \K^n$ such that $\mathcal{S}=S_K$. See, for example, \cite{Boeroeczky_LYZ_log_Mink, Boeroeczky_LYZ_Zhao_Gauss, Gardner_Hug_Weil_Xing_Ye_2019, Haberl_LYZ, OPstuff, Zhao_Gauss, Ivaki_Milman_Uniqueness, DylLyu, LivMi, Mui_Aleksandrov} for some recent developments and advances.
	
	In \cite{CorKlar, RotSur1}, functional versions of this problem were considered and partially solved for the measure $\mathcal{M}_u$ in \eqref{eq:intro_rot2}: For a given Borel measure $\mathcal{M}$ on $\R^n$, find sufficient and necessary conditions so that there exists $u \in \Conv(\R^n)$ such that $\mathcal{M}=\mathcal{M}_u$. Further generalizations were explored in \cite{dual_curvature_log, dual_curvature_orl}. One of the fundamental steps in these works was finding a functional version of the Wulff shape: For a continuous function $\zeta:\R^n \to \R$ with compact support and $u \in \Conv(\R^n),  t \in \R$, consider the perturbation
	\begin{equation}\label{eq:intro_Leg}
		(u^*+t\zeta)^*.
	\end{equation}
	A functional version of \eqref{eq:intro_Al} can be found in \cite[(42)]{CorKlar} (later extended in \cite[Lemma 5.2]{dual_curvature_log} to more general measures). In \cite{RotSur1}, lighter hypotheses were considered on $\zeta$, which was required to be bounded and continuous. In all these cases, the first variation under the perturbation \eqref{eq:intro_Leg} reads as
	\begin{equation}\label{eq:intro_al_k}
		\left. \frac{d}{dt} \mu_n((u^*+t\zeta)^*)\right|_{t=0}=\int_{\dom(u)}\zeta(\nabla u(x))e^{-u(x)}\, dx=\int_{\R^n}\zeta(y)\, d\mathcal{M}_u(y),
	\end{equation}
	where $u \in \Conv(\R^n)$ is coercive, that is, $\lim_{|x|\to \infty}u(x)=+\infty$.
	
	A striking difference between the correspondences \eqref{eq:intro_c}-\eqref{eq:intro_rot2} and \eqref{eq:intro_Al}-\eqref{eq:intro_al_k}, is that \eqref{eq:intro_al_k} lacks the boundary term appearing in \eqref{eq:intro_rot2}. This is where our work starts: We propose a method to compute the first variation of $\mu_n((u^*+t\zeta)^*))$ for a family of functions which allows the appearance of a boundary term in \eqref{eq:intro_al_k}. In particular, this is achieved for $u$ in \[\Coco\coloneqq \{u \in \Conv(\R^n): \dom(u)\text{ is compact}\} \] and $\zeta$ in \[C_{rec}(\R^n)\coloneqq \{ \zeta \in C(\R^n) : \rho_\zeta \text{ is continuous and } |\rho_\zeta -\zeta|\text{ is bounded}\} \] where $\rho_\zeta$ is the recession function of $\zeta$, defined as \[\rho_{\zeta}(\Nu)\coloneqq\lim_{t \to +\infty} \frac{\zeta(t\Nu)}{t}, \quad \Nu \in \mathbb{S}^{n-1},\] and then extended homogeneously on $\R^n$. Notice that if $v \in \Conv(\R^n) \cap \Coco$ then $v^* \in C_{rec}(\R^n)$. Therefore, by the properties of the Fenchel-Legendre transform, our approach includes perturbations by infimal convolution, covering the case \eqref{eq:intro_fvar} in $\Coco$. We note that the choice of these families of functions is an inheritance vice of the method at hand. Indeed, our strategy will be to approximate the epi-graph of $(u^*+t\zeta)^*$ via a sequence of compact Wulff shapes (see Section 3 for more details) with specific properties, allowing a limit process which might not be possible in general. Yet, $C_{rec}(\R^n)$ extends substantially the family of suitable perturbations since $\zeta \in C_{rec}(\R^n)$ does not need to have compact support, nor to be bounded. See Remark \ref{remark:domain} for further details on the choice of $C_{rec}(\R^n)$.
 
    Our main result, which will be later exposed in wider generality (that is, for a whole class of functionals including \eqref{eq:logconc_vol}) in Theorem \ref{gen_al}, reads as follows.
	\begin{theorem}\label{thm:intro_gen}
		Let $u \in \Coco$ and $\zeta \in C_{rec}(\R^n)$. Then, 
			\begin{equation}\label{eq:intro_mine}
			\begin{aligned}
				\left. \frac{d}{dt} \mu_n((u^*+t\zeta)^*)\right|_{t=0}=&\int_{\dom(u)}\zeta(\nabla u(x))e^{-u(x)}\, dx + \int_{\partial \dom(u)}\rho_\zeta(N_{\dom(u)}(y))e^{-u(x)}\, d\haus^{n-1}(y)\\=&\int_{\R^n}\zeta(y)\, d\mathcal{M}_u(y)+\int_{\sph^{n-1}}\rho_\zeta(\Nu)\, d\mathcal{S}_u(\Nu).
			\end{aligned}
		\end{equation}
	\end{theorem}\noindent
	In Theorem \ref{gen_al} we will prove that a variational formula in the spirit of \eqref{eq:intro_mine} still holds considering different measures on $\R^{n+1}$ other than $\mu_n$, under suitable summability assumptions. This is achieved through some recent results from the weighted Brunn-Minkowski theory. See Section 3 for more details and references. As pointed out later in Remark \ref{remark:HJ}, this kind of result can be interpreted as a Reynolds-type transport theorem under the flow induced by a Hamilton-Jacobi equation.
	
	As a byproduct, we can extend \eqref{eq:intro_mine} in the context of the singular integrals considered in \cite{dual_curvature_log}. In particular, we can remove some hypotheses in the setting of coercive convex functions, which were needed in \cite{dual_curvature_log}, and obtain further improvements in the restricted frame of $\Coco$ for \cite[Theorem 1.3]{dual_curvature_log}. In particular, for $q>0$, we obtain a formula for the one-sided derivative of the functional \[\mu_q(u)\coloneqq \int_{\R^n} e^{-u(x)}|x|^{q-n}\, dx, \quad u \in \Coco. \]
	\begin{theorem}
		\label{thm:dual_new_intro}
		Let $u,v \in \Conv(\R^n)$ and choose $q>0$. If $0<\mu_q(u)<+\infty$, $v \in \Coco$, and $v(0),u(0)<+\infty$, then,
		\begin{align*}
			&	\left. \frac{d}{dt}\mu_q(u\square (t\sq v))\right|_{t=0^{+}}=\\ \int_{\R^n} v^*(\nabla u(x)) e^{-u(x)}|x|^{q-n}\, dx + &\int_{\partial \dom(u)}  h_{\dom(v)}(N_{\dom(u)}(y)) e^{-u(y)}|y|^{q-n}\,d\haus^{n-1}(y).
		\end{align*}
	\end{theorem}
	
	This work is structured as follows. In Section 2, we present some preliminaries on convex sets and convex functions. In Section 3, we justify the notion of functional Wulff shape and, after some technical results on weighted surface area measures, we prove our main result in Theorem \ref{gen_al}, from which \eqref{eq:intro_mine} will follow. Finally, in Section 4, we specialize our techniques to obtain Theorem \ref{thm:dual_new_intro}.

	\section{Preliminaries}\label{Prelim1}
	
	The ambient space where we work is the Euclidean space $\R^n, n \in \N$. At times, we switch the point of view to $\R^{n+1}$, and for the standard basis $\{e_1,\ldots,e_{n+1}\}$ we identify $\R^n$ with the hyperplane $H$ orthogonal to $e_{n+1}$. The origin is denoted by $o$ in both cases. For $x,y \in \R^n$ we denote their scalar product by $x\cdot y$ and write $|x|=\sqrt{x\cdot x}$. The dimension of the ambient space will always be clear from the context, and we use the uppercase $X, Y$ when working in $\R^{n+1}$. Given a set $A \subset \R^n$, $\mathrm{int}A$ denotes its interior.
	
	\subsection{Convex sets}
	We start by summarizing some notions and results from the classical theory of convex sets.
	Of particular interest is the family of non-empty compact convex sets of $\R^n$, denoted by $\K^n$, endowed with the Hausdorff metric. The family of convex bodies, i.e., the elements of $\K^n$ with non-empty interior, is denoted $\K^n_n$. For an exhaustive exposition on the topic, see, for example, Schneider's monograph \cite{schneider_2013}.
	
	To every convex sets $K$ we can associate a function to $\R^n$, called support function, defined as \[h_K(x)\coloneq \sup\{x\cdot y : y \in K\}, \, x \in \R^n. \] Clearly, $h_K$ is finite if and only if $K  \in \K^n$. The behaviour of these functions with respect to Minkowski addition is encoded in the following property: for every convex set $K, L \subset \R^n$ \[h_K+h_L=h_{K+L}.\] A further representation of a convex set $K \subset \R^n$ is given by its indicator function, defined as
	\begin{align*}
		I_K(x)\coloneq
		\begin{cases}
			0 \qquad &\text{  if }x \in K,\\
			+\infty \quad &\text{ otherwise}.
		\end{cases}
	\end{align*} 
	Another class of functions representing a family of convex sets is the one of radial functions. If $K \in \K^n$ is such that $o \in \mathrm{int}K$, the radial function of $K$ is defined as \[r_K(x)\coloneqq \max \{s>0 : sx \in K\}, \, x\in \R^n\setminus\{o\}.\]
	
	Let $K \in \K^n$ and consider its boundary $\partial K$. When $K$ is a $C^2_+$ body, the Gauss map \[N_K: \partial K \to \mathbb{S}^{n-1}\] gives for each $x \in \partial K$ the unit normal vector to $\partial K$ at $x$, and is well defined and bijective. In general, for every $K \in \K^n$, the map $N_{K}$ is defined $\haus^{n-1}$-almost everywhere on $\partial K$. In the points $x$ where this is not single-valued, we consider $N_K(x)$ as the unit normal cone of $K$ at $x$, or generalized Gauss map. This is the set of all the unit vectors $\Nu$ such that $K$ has a tangent hyperplane at $x$ with outer unit normal $\Nu$. Whenever $N_K(\Nu)$ is well defined at $r_K(\Nu)\Nu\in \partial K$, we write $\alpha_K(\Nu)$ for $N_K(r_K(\Nu)\Nu)$. By \cite[Lemma 2.2]{OPstuff}, if a sequence of sets $K_m \in \K^n$ converges to $K \in \K^n$ with respect to the Hausdorff metric, then $\alpha_{K_m}$ converges to $\alpha_K$ almost everywhere on $\sph^{n-1}$ with respect to the spherical Lebesgue measure.
	
	The surface area measure of $K$ is defined as 
	\begin{equation}\label{surfacearea}
		S_K(B)\coloneqq \haus^{n-1}(N_K^{-1}(B))
	\end{equation}
	for every Borel set $B \subset \mathbb{S}^{n-1}$, where $\haus^{n-1}$ is the $(n-1)$-dimensional Hausdorff measure. Surface area measures are finite Borel measures on $\mathbb{S}^{n-1}$, and they are weakly continuous (see \cite[Section 4.2]{schneider_2013}). A sequence $(\mu_m)$ of Borel measures on $\sph^{n-1}$ converges weakly to a Borel measure $\mu$ if, for every $f\in C(\sph^{n-1})$, \[\int_{\sph^{n-1}}f\, d\mu_m \to \int_{\sph^{n-1}}f\, d\mu.\]
	
	\subsection{Convex functions}
	
	For a function $u:\R^n \to \R\cup \{+ \infty\}$, its epigraph is the set \[ \epi(u)\coloneqq \{(x,z)\in \R^{n+1}: u(x)\leq z\}.\] We say that $u$ is a convex function if $\epi(u)$ is a convex set. We work in the space \[\Conv(\R^n)\coloneq\{u:\R^n \to \R\cup \{+ \infty\}\colon u\text{ is convex, lower semi-continuous, } u \not \equiv +\infty\}.\]
	The domain of a convex function is the set \[ \dom(u)\coloneq\{ x \in \R^n: u(x)<+\infty\}. \] 
	
	On the space $\Conv(\R^n)$, we consider the topology induced by epi-convergence, characterized as follows: A sequence of functions $u_j \in \Conv(\R^n)$ epi-converges to $u \in \Conv(\R^n)$ if for every $x \in \R^n$, the following conditions are satisfied.
	\begin{itemize}
		\item For every sequence of points $x_j \in \R^n$ converging to $x \in \R^n$, $u(x)\leq \liminf_{j \to \infty} u_j(x_j)$.
		\item There exists a sequence $(x_j)$ converging to $x$, such that $u(x)=\lim_{j \to \infty} u_j(x_j)$.
	\end{itemize}
	
	Consider the subfamily \[\Cosc(\R^n)\coloneq\{u \in \Conv(\R^n): \lim_{|x|\to \infty} u(x)=+\infty \} \] of convex coercive functions. Here, convergence of level sets gives a more intuitive characterization of epi-convergence. For $u \in \Conv(\R^n), t \in \R \cup \{+\infty\}$ consider the sublevel set \[\{u\leq t\}\coloneq\{x \in \R^n:u(x)\leq t\}. \] For a sequence of functions $u_j \in \Conv(\R^n)$ we use the convention $\{u_j \leq t\} \to \emptyset$ if there exists $j_0 \in \N$ such that $\{u_j \leq t\}=\emptyset$ for every $j \geq j_0$.
	\begin{lemma}[{\cite[Lemma 5]{cocofu}}]\label{epiconv}
		A sequence of functions $(u_j) \subset \Cosc(\R^n)$ epi-converges to $u \in \Cosc(\R^n)$ if and only if $\{u_k \leq t\} \to \{u \leq t\}$ in the Hausdorff metric for every $t \in \R$ with $t \neq \min_{x \in \R^n} u(x)$.
	\end{lemma}
	
	For $u \in \Conv(\R^n)$, its Fenchel-Legendre transform is the function \[u^{*}(x)\coloneq\sup_{y \in \R^n} \{x\cdot y - u(y)\}. \] Note that since it is the supremum of affine (and thus convex) functions, $u^{*}$ is again a convex function. Moreover, for a non-convex function $f: \R^n \to \R$, it still makes sense to define $f^{*}$ as above if we admit the trivial case $f^* \equiv +\infty$. In particular, $f^*$ is convex. Moreover, when $u \in \Conv(\R^n)$ one has the important fact \[ (u^{*})^{*}=u.\] 
	
	In $\Conv(\R^n)$, we consider the operations of pointwise addition and infimal convolution. For $u,v \in \Conv(\R^n)$ the latter is defined as \[u \square v (x)\coloneq \inf_{y \in \R^n} \{u(y)+v(x-y)\}, \] and $u \square v \in \Conv(\R^n)$ if at least one of the two functions is coercive (which will always be our case). For $u \in \Conv(\R^n), t \geq 0$, instead of the dilation, we have the two following notions of multiplication: The classical scalar multiplication and the \text{epi-multiplication} \[(t \sq u)(x)\coloneq t u \left( \frac{x}{t}\right) \] for $t>0$, $0\sq u=I_{\{0\}}$. Moreover, one has the fundamental relation
	\begin{equation}\label{dualsum}
		((t\sq u) \square (s\sq v))^{*}=tu^{*}+sv^{*}
	\end{equation}
	for $u,v \in \Conv(\R^n)$ and $s,t \geq 0$.
	Notice that for $u,v \in \Conv(\R^n)$ and $t>0$, $\epi(u \square v)$ is the Minkowski sum $\epi(u)+\epi(v)$, while $\epi(t \sq v)=t(\epi(v))$.
	
	A further meaningful correspondence is the one between support functions and indicator functions of compact convex sets.
	Indeed, if $K \in \K^n$ and $K$ contains the origin, a quick calculation shows that $(I_K)^{*}=h_K$. 
	
	\subsection{The space $\Coco$}\label{coco}
	
	In our investigation, a fundamental role is played by the space \[\Coco\coloneqq\{u\in \Cosc(\R^n): \dom(u) \text{ is compact} \}, \] which is a subset of $\Cosc(\R^n)$. Both these spaces are endowed with the topology induced by epi-convergence introduced earlier. The results and notions exposed in the remainder of this section are from the author and Knoerr \cite{FromConvToFunc2023}, where the family $\Coco$ was introduced as a tool to infer geometric properties of convex functions through local properties of compact convex sets. 
	
	These functions can be obtained from compact convex sets in $\R^{n+1}$ using the following construction where we use the identification $H\equiv \R^n=e^\perp_{n+1}$: To every $K \in \K^{n+1}$ we associate the function $\lfloor K \rfloor:\R^n\rightarrow \R\cup\{+\infty\}$ defined by
	\begin{equation}\label{eq:low_bound_func}
		\lfloor K \rfloor (x)\coloneqq
		\inf\{t\in\R: (x,t)\in K\}.
	\end{equation}
	In addition, $\lfloor K\rfloor(x)=+\infty$ if and only if $(x,t)\notin K$ for all $t\in\R$. Analogously, for every $x \in \R^n$ and $K \in \K^{n+1}$ we can define the concave function 
	\begin{equation*}
		\lceil K \rceil (x)\coloneqq
		\sup\{t\in\R: (x,t)\in K\}.
	\end{equation*}
	In this case $\lceil K\rceil (x)=-\infty$ if $(x,t) \notin K$ for all $t \in \R$. The following results are proved only for the map $\lfloor \cdot \rfloor$ for the sake of brevity, but they also hold (up to the expected adjustments) for the map $\lceil \cdot \rceil$ since for $K \in \K^n$ one has $\lceil K \rceil (x)=-\lfloor R_H K \rfloor (x)+c$, where $R_{H}:\R^{n+1}\to \R^{n+1}$ is the reflection with respect to $H$ and $c$ is a suitable constant. The following properties can be found, for example, in \cite{FromConvToFunc2023}.
	\begin{lemma}\label{lemma:cont_bod}
		For every $K\in\mathcal{K}^{n+1}$, $\lfloor K\rfloor\in \Coco$. Moreover, the map $\lfloor \cdot\rfloor:\mathcal{K} ^{n+1}\rightarrow\Coco$ is continuous.
	\end{lemma}
	
	Conversely, we may associate to any $u\in \Coco$ a convex set in the following way: We set $M_u\coloneqq\max_{x \in \dom(u)} u(x)$, which is finite since the domain of $u$ is compact and $u$ is convex, and define
	\begin{equation}\label{eq:func_to_bod}
		K^u\coloneqq \epi (u-M_u) \cap R_{H}(\epi (u-M_u))+M_u e_{n+1},
	\end{equation}
	where $R_H$ is the reflection with respect to $H$. This is a non-empty compact and convex set, i.e., $K^u \in \K^{n+1}$. We thus obtain a map
	\begin{equation}\label{ku}
		\begin{gathered}
			\Coco \to \K^{n+1} \\ \qquad \qquad u \mapsto K^u.
		\end{gathered}
	\end{equation}
	By construction $u=\lfloor K^u \rfloor$ for $u\in\Coco$. This approach resembles the one adopted in \cite{Klar_mar}, which in turn stems from \cite{Func_Sant}.
	
	Consider the lower half-sphere $\sph^n_{-}\coloneqq \{\Nu \in \sph^n: \Nu\cdot e_{n+1}<0\}$. We define the lower boundary of $K$ by 
	\begin{align*}
		\partial K_{-}\coloneqq \{X \in \partial K: N_K(X) \cap \sph^n_{-}\neq \emptyset \}.
	\end{align*}
	Notice that the graph of $\lfloor K \rfloor$ coincides with the closure of $\partial K_{-}$.
	In particular, if $u\in\Coco$, then the closure of $\partial K^u_-\subset \R^{n+1}$ is the graph of $u$,
	and we can parameterize it via the map
	\begin{align*}
		f_u:\dom (u) &\to \R^{n+1}\\
		x&\mapsto (x,u(x)).
	\end{align*} 
	If $\gamma:\R^{n+1}\rightarrow\R$ is bounded and Borel measurable, then by the area formula \cite[Theorem 8.1]{Maggi_finite}, we obtain
	\begin{equation}\label{change1}
		\int_{\partial {K^u_-}} \gamma(X)\,d\haus(X)=\int_{\dom(u)} \gamma\left((x,u(x))\right)\sqrt{1+|\nabla u(x)|^2} \,dx, 
	\end{equation}
	where $\sqrt{1+|\nabla u(x)|^2}$ is the approximate Jacobian of $f_u$. \\
	
	If $u\in\Coco$ is differentiable at $x\in\dom(u)$, then $\frac{(\nabla u(x),-1)}{\sqrt{1+|\nabla u(x)|^2}}$ is the unique unit outer normal vector to $\epi(u)$ at $(x,u(x))$. Since $u$ is convex, it is differentiable almost everywhere, and thus, the unit normal vectors to the epigraph are defined almost everywhere. We have the following.
	\begin{lemma}[\!{\cite[Corollary 3.7]{FromConvToFunc2023}}]
		\label{lemma:change3}
		For every $u\in\Coco$ and $\eta \in C(\sph_-^n)$
		\begin{align}
			\label{eq:change3}
			\int_{\sph_-^n} \eta(\Nu)\,dS_{K^u}(\Nu)=\int_{\dom(u)} \eta\left(\frac{(\nabla u(x),-1)}{\sqrt{1+|\nabla u(x)|^2}}\right)\sqrt{1+|\nabla u(x)|^2} \,dx.
		\end{align}
	\end{lemma}

	We conclude this section by recalling the parallelism between support functions and the Fenchel-Legendre transform. Via explicit calculations, we infer
	\begin{equation}\label{esupp}
		\begin{aligned}	
			h_K(y,-1)=&\sup\{X \cdot (y,-1)\colon X \in K\}=\sup\{X \cdot (y,-1)\colon X \in \partial K\}\\=&
			\sup\{(x,\lfloor K \rfloor(x))\cdot (y,-1)\colon x \in \dom(\lfloor K \rfloor)\}\\=&\sup\{x\cdot y-\lfloor K \rfloor (x) \colon x \in \dom(\lfloor K \rfloor)\}\\=&\sup\{x\cdot y-\lfloor K \rfloor (x) \colon x \in \R^n\}={ \lfloor K \rfloor}^{*}(y).
		\end{aligned}
	\end{equation}
	When $K=K^u$ for some $u \in \Coco$, \eqref{esupp} takes the form
	\begin{equation}\label{esupp2}
		h_{K^u}(y,-1)=u^{*}(y)
	\end{equation}
	The map \[ K \mapsto h_K(\cdot,-1)\] was already considered, for example, by Alesker \cite{Alesker} and Knoerr \cite{Knoerrsupportduallyepi2021} to create a correspondence between $\K^{n+1}$ and the space of finite convex functions. Equation \eqref{esupp} shows that the latter point of view and the one presented here are dual.	
	
	\section{First variations for measures of epigraphs}
	
	Consider on $\R^{n+1}$ a measure $\mu$ such that \[d\mu(x,z)=d\eta(x)d\omega(z), x \in \R^n\equiv H, z \in \R\equiv H^{\perp},\] where $\eta$ and $\omega$ are positive Borel measures on $\R^n$ and $\R$ respectively.
	The space of interest is always some subset of the family of functions $\Conv(\R^n)$, and we want to evaluate the measure of the epigraph of $u$ through $\mu$, that is (using the abuse of notation $\mu(u)\coloneqq \mu(\epi(u))$) \[\mu(u)=\int_{\epi(u)}\,d\mu(z,x)=\int_{\dom(u)}\int_{u(x)}^{+\infty} \,d\omega(z)d\eta(x). \] Ignoring, for now, the various summability assumptions, if we define $\Phi(t)=\omega([t,+\infty))$ we can write \[\mu(u)=\int_{\dom(u)}\Phi(u(x))\,d\eta(x). \]
	We focus on the case where there exist $\phi \in C(\R)$ and $\psi \in C(\R^n)$ such that $d\omega(z)=\phi(z)dz$ and $d\eta(x)=\psi(x)dx$. We now introduce a suitable deformation process for functions in $\Coco$, which we build upon the classical notion of Wulff shape.
	
	\subsection{Wulff shapes of convex functions}\label{functional_wulff}
	
	The concept of Wulff shape, introduced more than a century ago by Wulff \cite{Wulff}, is nowadays a well-established scientific tool, especially in the study of the shapes of crystals. Significant developments have been obtained throughout the 20th century from the mathematical perspective. See, for example, the work of Fonseca \cite{Fon}. 
	\begin{definition}\label{def:Wulff_shape}
		Consider a function $f:\mathbb{S}^{n-1}\to \R$. Its Wulff shape is the set \[[f]\coloneq \bigcap_{\Nu \in \mathbb{S}^{n-1}} H^{-}_{\Nu}(f(\Nu)), \] where $H^{-}_{\Nu}(t)\coloneq\{x \in\ \R^n\colon \Nu \cdot x \leq t\}$ is the negative closed half-space with outer normal $\Nu$ and distance $t$ from the origin. Equivalently, $[f]$ is the unique maximal (with respect to inclusion) convex set satisfying the condition 
		\begin{equation}\label{eq:equiv_wulff}
			h_{[f]}(\Nu)\leq f(\Nu) \text{ for every } \Nu \in \sph^{n-1}.
		\end{equation}
	\end{definition}
	\begin{remark}\label{rem:tr_Wulff}
		Notice that if $f>c>0$, $[f]$ is clearly non-empty. In general, if $\ell_y(x)\coloneq y \cdot x, y \in \R^n$ and $f-\ell_y>c>0$, then $[f]$ is non-empty and $y$ is in the interior of $[f]$. In particular, if $[f]$ is non-empty,  
		\begin{equation}\label{eq:trans_wulff}
			[f+\ell_y]=[f]+y,
		\end{equation}
		for every $y \in \R^n$ (notice that $\ell_y=h_{\{y\}}$). Indeed, by \eqref{eq:equiv_wulff}, for every $y \in \R^n$ \[h_{[f]+y}=h_{[f]}+\ell_y \leq f+\ell_y.\] If, by contradiction, $h_{[f]+y}$ was not maximal for $f+\ell_y$, neither would be $h_{[f]}$ for $f$, which would contradict \eqref{eq:equiv_wulff}, proving \eqref{eq:trans_wulff}.
	\end{remark}
	
	  Recall that we consider the admissible family of perturbations  
	\[C_{rec}(\R^n)\coloneqq \{ \zeta \in C(\R^n) : \rho_\zeta \text{ is continuous and } |\rho_\zeta -\zeta|\text{ is bounded}\},\] where \[\rho_{\zeta}(\Nu)\coloneqq \lim_{t \to +\infty} \frac{\zeta(t\Nu)}{t}, \quad \Nu \in \mathbb{S}^{n-1}.\] As for support functions, we use $\rho_\zeta$ to denote both a function on $\sph^{n-1}$ and its $1$-homogeneous extension given by $\rho_\zeta(x)= |x| \rho_\zeta(x/|x|), x \in \R^n$.
    We now show how, for $u \in \Coco$ and $\zeta \in C_{rec}(\R^n)$, the perturbation \[u_t\coloneq (u^*+t\zeta)^*\] can be represented, for $t>0$ sufficiently small, by means of a Wulff shape.
	
	For $\zeta \in C_{rec}(\R^n)$ we consider the function $\bar{\zeta}$ on $\mathbb{S}^{n}_{-}$ obtained as
	\begin{equation}\label{eq:zeta_proj}
		\bar{\zeta}(\Nu)\coloneq \frac{\zeta\left(g(\Nu)\right)}{\sqrt{1+\left|g(\Nu)\right|^2}}.
	\end{equation}
	Here $g: \mathbb{S}^n_{-}\to \R^n$ is the gnomonic projection
	\begin{equation}\label{gnomonic_projection}
		\Nu \mapsto \frac{\Nu-(e_{n+1}\cdot \Nu)e_{n+1}}{e_{n+1}\cdot \Nu}. 
	\end{equation}
	We will make use of the extension of $\bar{\zeta}$ on the whole $\mathbb{S}^{n}$ obtained by reflection on $H$, that is, if $\Nu=(\nu_1,\dots,\nu_{n+1}) \in \sph^n_{-}$, $\bar{\zeta}(\Nu)=\bar{\zeta}((\nu_1,\dots,\nu_n,-\nu_{n+1}))$. When $\nu_{n+1}=0$, $\bar{\zeta}$ is extended by continuity. This extension exists and is finite since $\zeta \in C_{rec}(\R^n)$. Note that with the identification $\mathbb{S}^n \cap H\equiv \mathbb{S}^{n-1}$ we have that $\bar{\zeta}$ restricted to $\mathbb{S}^{n-1}$ is equal to $\rho_\zeta$. To make the notation lighter, we refer to $\bar{\zeta}$ both for the transform \eqref{eq:zeta_proj} and the extension. 
 
    We will use this construction to work with the Wulff shape (in $\R^{n+1}$) obtained as \[ [h_{K^u}+t \overline{\zeta}]\] and its connection with $(u^*+t\zeta)^*$, which we investigate in this subsection. First, we show that the condition $\zeta \in C_{rec}(\R^n)$ is sufficient to guarantee that $(u^*+t\zeta)^* \in \Coco$. In fact, we prove more.
	\begin{lemma}\label{lemma:domain}
		Let $u \in \Coco$ and $\zeta \in C_{rec}(\R^n)$. Then, for every  $t \in [0,\varepsilon]$ for $\varepsilon>0$ sufficiently small,
		\[ \dom (u_t)=[h_{\dom(u)}+ t \rho_{\zeta}].\]
	\end{lemma}
	\begin{proof}
		Since $u \in \Coco$, we have that $m\coloneqq \min_{x \dom(u)} u(x)$ and $M\coloneqq \max_{x \in \dom(u)} u(x)$ exist and are finite. Therefore, \[ I_{\dom(u)}+m\leq u \leq I_{\dom(u)}+M,\] which, by the properties of the Fenchel-Legendre transform, is equivalent to 
		\begin{equation}\label{eq:conv_rec}
		 h_{\dom(u)}-M \leq u^* \leq h_{\dom(u)}-m.
		\end{equation}
		Choose $\varepsilon>0$ such that $u_t$ is proper for every $t \in [0, \varepsilon]$. Then \[h_{\dom(u)}-M+t(\rho_\zeta-C)\leq u^*+t\zeta\leq h_{\dom(u)}-m+t(\rho_\zeta+C).\] By the properties of Wulff shapes (see, for example, \cite[Proposition 1.2]{Fon}) \[(h_{\dom(u)}+t \rho_\zeta)^*=I_{[h_{\dom(u)}+ t \rho_{\zeta}]}.\] Therefore, we infer \[I_{[h_{\dom(u)}+ t \rho_{\zeta}]}+m-tC\leq u_t\leq I_{[h_{\dom(u)}+ t \rho_{\zeta}]}+M+tC,\] concluding the proof.
	\end{proof} \noindent
	The following lemma will be crucial in the proof of our main result.
	\begin{lemma}\label{lemma:ex_claim}
		In the hypotheses of Lemma \ref{lemma:domain} there exists $T>0$ such that \[[h_{K^u+\ell_{T+\tau}}+t\overline{\zeta}]=[h_{K^u+\ell_{T}}+t \overline{\zeta}]+\ell_\tau\] for every $\tau>0$.
	\end{lemma}
	\begin{proof}
		First,notice that for every $v: \R^n \to \R$ \[\epi(v^*)=\bigcap_{\Nu \in \sph^n_-}H^-_{\Nu} \left( \frac{v(g(\Nu))}{\sqrt{1+|g(\Nu)|^2}}\right).\] Then, for $T\geq 0$ \[[h_{K^u+\ell_{T}}+t\overline{\zeta}]=\epi(u_t)\cap \epi(-u_t+c+T) \cap [h_{\dom(u)}+ t \rho_{\zeta}]\times \R\] for a suitable constant $c \in \R$. 
		
		By Lemma \ref{lemma:domain} $u_t$ admits a global finite maximum on its domain (since the latter is compact). In particular, choosing $T$ large enough (independently on $t \in [0,\varepsilon]$ by compactness) $u_t < -u_t+c+T$, and \[[h_{K^u+\ell_{T}}+t\overline{\zeta}]=\epi((u^*+t\zeta))\cap \epi(-u_t+c+T),\] since they share the same domain. Noticing that $\ell_{T+\tau}=\ell_T+ \ell_\tau$ for $\tau>0$, standard computations (see, for example, \cite[Lemma 3.1.1]{schneider_2013}) conclude the proof.
	\end{proof}\noindent
	We therefore obtain the following immediate corollary, which justifies the terminology of functional Wulff shape.
	\begin{corollary}\label{cor:wulff}
		Let $u \in \Coco$, $\zeta \in C_{rec}(\R^n)$, and $t \in [0,\varepsilon]$ for $\varepsilon>0$ sufficiently small. Then, for every $T>0$ sufficiently large, \[u_t=\lfloor [h_{K^u+\ell_{T}}+t \overline{\zeta}]  \rfloor.\]
	\end{corollary}\noindent
    \begin{remark}\label{remark:domain}
        Our strong assumptions on $\rho$ are required precisely to obtain Corollary \ref{cor:wulff}. Indeed this guarantees by Lemma \ref{lemma:domain} a control over the evolving domain of $u_t$. Without this assumption, pathological phenomena could arise where $\dom(u_t)$ is not compact for every $t>0$. In the explicit calculations for the first variation of $\mu_n(u_t)$, this would imply that the boundary term would appear on the derivative only for $t=0$. Since we use a continuity argument to infer our formulas, we are not currently able to tackle this difficulty.
    \end{remark}
	
	The core ideas for the proof of our main result are encoded in the following properties of Wulff shapes proved by Willson \cite{WulffSemigroup}. Consider $K \in \K^{n+1}, t\geq0$ and $f \in C(\mathbb{S}^{n})$; we use the notation $F_t K$ for the Wulff shape of $h_K+t f$, that is 
	\begin{equation}\label{eq:wulff_alternative}
		F_t K\coloneqq [h_K+tf],
	\end{equation}
	The dependence on a function $f$ will always be clear from the context. Theorems 5.1 and 5.6 from \cite{WulffSemigroup} read as follows. 
	\begin{theorem}\label{continuity}
		If $K_m \to K$ in $\K^{n+1}$, $t_m \to t_0$ in $\R$ and $F_{t_0} K$ has non-empty interior, then $F_{t_m} K_m$ has non-empty interior for $m$ large and $F_{t_m} K_m \to F_{t_0} K$ in $\K^{n+1}$.
	\end{theorem}
	In particular, Theorem \ref{continuity} implies that $F_t K$ is continuous in $t$.
	\begin{theorem}\label{semigroup}
		Let $s$ and $t$ be nonnegative real numbers. Let $K \in \K^{n+1}_{n+1}, f \in C(\mathbb{S}^{n-1})$ and assume $F_t K$ has non-empty interior. Then \[F_s F_t K=F_{s+t} K\]
	\end{theorem}
	These properties will be very useful later to obtain differentiability in $t$ for the measure of $F_t K$, since by Lemma \ref{lemma:cont_bod}, all these properties are passed on to functional Wulff shapes. 

    \begin{remark}\label{remark:HJ}
        We conclude this section with the following observation. It is well-known (see, for example, \cite[Theorem 2.5]{HopfLSC}), that if $u \in \Conv(\R^n)$ and $\zeta \in C(\R^n)$, then \[w(t,x)\coloneqq (u^*+t\zeta)^*(x)\] is a viscosity solution (in the Crandall-Lions sense) of the Hamilton-Jacobi equation
            \begin{align*}
                \begin{cases}
                    \frac{\partial}{\partial t}w(t,x)+\zeta(\nabla_x w(t,x))=0, \quad &(t,x) \in  \mathrm{int}(\dom(w(t,x))),\\
                    w(0,x)=u(x) \quad &  x \in  \dom(u).
                \end{cases}
            \end{align*}
            As such, it enjoys many properties, including a semi-group property of the type in Theorem \ref{semigroup}. 
            Under our additional assumptions, we are able to describe precisely $\dom(w(t,x))$ thanks to Lemma \ref{lemma:domain}. We could use the previous result to compute pointwise the derivative of $w(t,x)$ on the interior of its domain, while separately computing the boundary term. Yet, to do this directly would require a non-smooth Reynolds-type transport theorem, which is precisely how Theorem \ref{gen_al} later can be interpreted. 
    \end{remark}
	
	\subsection{Weighted Brunn-Minkowski theory}\label{sec:measure_theoretic}
	
	Formula \eqref{eq:intro_mine} is an application of our results, which are built on a wider framework. Indeed, \eqref{eq:logconc_vol} can be interpreted as the measure $\mu$ on $\R^{n+1}$ of the epigraph of $u$, where $d\mu(x,z)=e^{-z}dxdz, x\in \R^n, z \in \R$. In recent years has drawn much interest the so-called weighted Brunn-Minkowski theory, where one considers instead of $\V$ a generic measure $\mu$ with continuous density $\Psi$ with respect to the Lebesgue measure. See, for example, Livshyts \cite{LivMi}, Alonso-Gutierrez, Hern\'andez Cifre, Roysdon, Yepes Nicolàs, and Zvavitch \cite{GutRoEtAl}, Kryvonos and Langharst \cite{DylLyu}, and Fradelizi, Langharst, Madiman and Zvavitch \cite{Weighted_I}. We note that the mentioned papers concern compact convex sets, while recently Schneider \cite{schneider_weighted} explored these ideas in the context of pseudocones, which, in our language, are epigraphs of functions in a suitable subfamily of $\Conv(\R^n)$. 
	
	We consider the following generalized notion of surface area measure in \eqref{surfacearea}, that is, the weighted surface area measure, which is defined as \[ S_{\mu, K}(B)=\int_{N_K^{-1}(B) }\Psi(X)\,d\mathcal{H}^n(X)\] for every $K \in \K^n$ and Borel set $B \subset \sph^n$. A core ingredient of our approach is Lemma 2.7 from \cite{DylLyu}, which generalizes Aleksandrov's variational lemma {\cite[Lemma 7.5.3]{schneider_2013}}.
	\begin{lemma}[Kryvonos and Langharst]\label{alex}
		Let $\mu$ be a Borel measure on $\R^{n+1}$ with continuous density $\Psi$ with respect to the Lebesgue measure. Then for $f \in C(\mathbb{S}^{n})$ and $K \in \K^{n+1}_{n+1}$, we have
		\begin{equation}\label{eq:variation_DylLyu}
			\left.\frac{d}{dt}\mu([h_K+tf])\right|_{t=0}=\int_{\mathbb{S}^{n}}f(\Nu)\,dS_{\mu,K}(\Nu)=\int_{\partial K} f(N_K(X)) \Psi(X)\, d\haus^n(X).
		\end{equation}
	\end{lemma}
	\begin{remark}\label{rem:tr_Al}
		Lemma \ref{alex} was originally formulated with the additional assumption $o \in \mathrm{int}K$, which was unnecessary since the boundary structure of a compact convex set is invariant under translations. Thus, both the set and the measure can be suitably translated so that \eqref{eq:variation_DylLyu} still holds. For completeness, and since it will be a useful remark later, we provide a proof of this fact. Suppose that the origin is not contained in $K$. For every point $Y$ in the interior of $K$, $h_{K}-\ell_Y>0$, where $\ell_Y(X)=Y \cdot X, X \in \R^{n+1}$, and $K-Y$ has the origin in its interior. Then, by \eqref{eq:trans_wulff}, we can consider the Wulff shape $[h_K-\ell_Y+tf]$ to infer by Lemma \ref{alex} \[\left.\frac{d}{dt}\mu([h_K-\ell_Y+tf])\right|_{t=0}=\int_{\partial (K-Y)} f(N_{K-Y}(X)) \Psi(X)\, d\haus^n(X).\] 
		Consider $\tilde{\Psi}(\cdot)=\Psi(\cdot +Y)$ and the measure $\tilde{\mu}$ which has $\tilde{\Psi}$ as density. By Lemma \ref{alex}, the translation invariance of the Hausdorff measure, and since $N_{K-Y}(X)=N_K(X+Y)$,
		\begin{align*}
			\left.\frac{d}{dt}\mu([h_K+tf])\right|_{t=0}=&\left.\frac{d}{dt}\tilde{\mu}([h_K-\ell_Y+tf])\right|_{t=0^+}=\\
			\int_{\partial (K-Y)} f(N_{K-Y}(X)) \tilde{\Psi}(X)\, d\haus^n(X)=&\int_{\partial K} f(N_{K}(Z)) \tilde{\Psi}(Z-Y)\, d\haus^n(Z)\\=& \int_{\partial K} f(N_{K}(Z)) \Psi(Z)\, d\haus^n(Z),
		\end{align*}
		proving that in Lemma \ref{alex}, we do not need the origin to belong to the interior of $K$.
	\end{remark}
	
	We now show that the weighted volume is continuous for compact convex sets with non-empty interior, and the weighted surface area measure is weakly continuous. The latter fact follows implicitly, for continuous densities, from the stability result by Livshyts \cite[Proposition A.3]{LivMi}. For our applications, we include the possibility of working with singular densities. Our ambient space is still $\R^{n+1}$ with the formalism introduced before, where we write $X \in \R^{n+1}$ as $X=(x,z), x \in \R^n\equiv H, z \in \R\equiv \mathrm{span}\{e_{n+1}\}$.
	\begin{lemma}\label{lemma:weak_continuity}
		For $q>0$ and $\Psi \in C(\R^{n+1})$ non-negative, consider the measure $\mu$ on $\R^{n+1}$ with density $d\mu(X)=\Psi(X)|x|^{q-n}dxdz$, and a sequence $(K_m) \subset \K^{n+1}$ converging to $K \in \K^{n+1}_{n+1}$ in the Hausdorff metric. In case $q<n$, suppose moreover that $o \in \mathrm{pr}_H\left(\mathrm{int} K \cap \bigcap_{m \in \N} \mathrm{int}K_m \right)$, where $\mathrm{pr}_H$ is the orthogonal projection on $H\equiv \R^n$. Then
		\begin{itemize}
			\item The volume is continuous (this conclusion is trivial for $q \geq n$), that is, \[\mu(K_m)\to \mu(K),\]
			\item The sequence of measures $S_{\mu,K_m}$ converges weakly to $S_{\mu,K}$.
		\end{itemize} 
	\end{lemma}
	\begin{proof}
		First, notice that without loss of generality, we can suppose $o \in \left(\mathrm{int} K \cap \bigcap_{m \geq M} \mathrm{int}K_m \right)$ for $M$ sufficiently large. Indeed, if $q\geq n$, we can follow Remark \ref{rem:tr_Al}. If $0<q<n$, thanks to the additional hypothesis $o \in \mathrm{pr}_H\left(\mathrm{int} K \cap \bigcap_{m \in \N} \mathrm{int}K_m \right)$ we just need a vertical translation, which will affect only $\Psi$ proceeding as in the same remark.
		
		Under this assumption, there exists $\delta>1$ such that
		\begin{equation*}\label{eq:unif_bodies}
			\frac{1}{\delta}B^n \subset K_m \subset \delta B^n
		\end{equation*}
		for every $m$ sufficiently large. This chain of inclusions is equivalent to
		\begin{equation}\label{eq:ineq_unif}
			\frac{1}{\delta}<r_{K_m}<\delta
		\end{equation}
		for $m$ every sufficiently large, and since $K_m \to K$ in the Hausdorff metric the sequence of radial functions $r_{K_m}$ converges uniformly to $r_K$.	Moreover, for $f \in C(\sph^n)$ we have that $\alpha_{K_m}(\Nu)\to \alpha_K(\Nu)$ for almost every $\Nu \in \sph^n$ with respect to the spherical Lebesgue measure.
		
		By direct computations using polar coordinates, 
		\begin{align*}
			\mu(K_m)=&\int_{\sph^n} \int_0^{r_{K_m}(\Nu)} \psi(t\Nu)|\mathrm{pr}_H(t\Nu)|^{q-n}t^n\, dtd\Nu\\=&\int_{\sph^n} |\mathrm{pr}_H(\Nu)|^{q-n} \left( \int_0^{r_{K_m}(\Nu)} \psi(t\Nu) t^q\, dt \right)d\Nu.
		\end{align*}
		The radial functions $r_{K_m}$ are continuous and converge uniformly to $r_K$. Therefore \[\int_0^{r_{K_m}(\Nu)} \psi(t\Nu) t^q\, dt \to \int_0^{r_{K}(\Nu)} \psi(t\Nu) t^q\, dt\] uniformly, and the continuity of $\mu$ is proved once we show that 
		\begin{equation}\label{eq:funny_integral}
			\int_{\sph^n}|\mathrm{pr}_H(\Nu)|^{q-n}\, d\Nu < +\infty.
		\end{equation}
		Consider indeed the representation \[\int_{\sph^n}|\mathrm{pr}_H(\Nu)|^{q-n}\, d\Nu=2\int_{\sph^n_{-}}|\mathrm{pr}_H(\Nu)|^{q-n}\, d\Nu.\] By \eqref{change1} \[\int_{\sph^n_{-}}|\mathrm{pr}_H(\Nu)|^{q-n}\, d\Nu=\int_{B^n}\frac{|x|^{q-n+1}}{\sqrt{1-|x|^2}}\, dx=\int_{\sph^{n-1}}\int_0^1\frac{t^{q}}{\sqrt{1-t^2}}\, dt d\Nu,\] which is finite by direct computations. Thus, we have \eqref{eq:funny_integral}.
		
		For the second statement, we adapt the argument in \cite[Theorem 3.4]{Zhao_Gauss}. By the definition of weighted surface area measure and \cite[Lemma 2.9]{OPstuff}, we infer 
		\begin{align*}
			\int_{\sph^n}f(\Nu)\, dS_{\mu,K_m}(\Nu)=\int_{\partial K_m} f(N_K(X))\Psi(X)|x|^{q-n}\, d\haus^n((x,z))\\=\int_{\sph^n} \frac{f(\alpha_{K_m}(\Nu))}{\Nu\cdot \alpha_{K_m}(\Nu)}\Psi(r_{K_m}(\Nu)\Nu )|\mathrm{pr}_H(r_{K_m}(\Nu)\Nu)|^{q-n}r_{K_m}^n(\Nu)\, d\Nu.
		\end{align*}
		By \eqref{eq:ineq_unif} and since $f$ and $\Psi$ are continuous, there exists $C>0$ such that 
		\[\frac{f(\alpha_{K_m}(\Nu))}{\Nu\cdot \alpha_{K_m}(\Nu)}\Psi(r_{K_m}(\Nu)\Nu )|\mathrm{pr}_H(r_{K_m}(\Nu)\Nu)|^{q-n}r_{K_m}^n(\Nu)\leq \delta^{n+1+(q-n)\mathrm{sign}(q-n)}C |\mathrm{pr}_H(\Nu)|^{q-n}\] for every $m$ and for almost every $\Nu \in \sph^n$, where $\mathrm{sign}(\tau)=\tau/|\tau|$ if $0\neq \tau \in \R$ and $\mathrm{sign}(0)=0$. Notice that the expression on the left converges almost everywhere to \[\frac{f(\alpha_{K}(\Nu))}{\Nu\cdot \alpha_{K}(\Nu)}\Psi(r_{K}(\Nu)\Nu )|\mathrm{pr}_H(r_{K}(\Nu)\Nu)|^{q-n}r_{K}^n(\Nu).\]
		
		By \eqref{eq:funny_integral} and a direct application of the dominated convergence theorem, the proof is concluded.
	\end{proof}\noindent
	Notice that the estimates provided in the proof imply that for every $K \in \K^{n+1}$ such that $o \in \mathrm{pr}_H(\mathrm{int} K)$ the values $\mu(K)$ and $S_{\mu, K}(\sph^n)$ are finite.
	
	We now apply Lemma \ref{lemma:weak_continuity} to show that Lemma \ref{alex} holds for the measures we are considering. For the convenience of the reader, we recall the following classical result (see, for example, \cite[Theorem 7.17]{baby_rudin}).
	\begin{lemma}\label{thm:lim_der}
		Suppose $f_m:[a,b]\to \R, m \in \N$ is a sequence of functions, differentiable on $[a,b]$ and such that $f_m(x_0)$ converges for some $x_0 \in [a,b]$. If the derivatives $f'_m$ converge uniformly on $[a,b]$, then $f_m$ converges uniformly on $[a,b]$, to a function $f$, and \[f'(x)=\lim_{m\to \infty} f'_m(x)\] for every $x \in [a,b]$.
	\end{lemma}\noindent
	We treat only the case $0<q<n$, since for $q\geq n$ the following is a direct consequence of Lemma \ref{alex}.
	\begin{corollary}\label{cor:alex}
		Let $\mu$ be a Borel measure on $\R^{n+1}$ with density $d\mu(X)=\Psi(X)|x|^{q-n}dxdz$, $0<q<n$. Then for $f \in C(\mathbb{S}^{n})$ and $K \in \K^{n+1}_{n+1}$ such that $o \in \mathrm{pr}_H(\mathrm{int}K)$, we have
		\begin{equation*}
			\left.\frac{d}{dt}\mu([h_K+tf])\right|_{t=0}=\int_{\mathbb{S}^{n}}f(\Nu)\,dS_{\mu,K}(\Nu).
		\end{equation*}
	\end{corollary}
	\begin{proof}
		Again, without loss of generality, we can suppose that $o \in \mathrm{int}K$. Consider the sequence of measures $\mu_m$ defined by $d\mu_m(x,z)=\Psi(X)(|x|\wedge m)^{q-n}$, where $a\wedge b\coloneqq \min\{a,b\}$ for $a,b \in \R$. Then, for every $m$, the density of $\mu_m(x,z)$ with respect to the Lebesgue measure is a continuous function on $\R^{n+1}$. Moreover, since $o \in \mathrm{int}K$, by Theorem \ref{continuity} there exist $\varepsilon>0$ such that $o \in \mathrm{int}[F_t K]$ for every $t\in [0,\varepsilon]$, where $F_t K=[h_K+tf]$. In particular, $\mu_m(K_t)$ is continuous in $[0,\varepsilon]$ for every $m$. 
		
		By Theorem \ref{semigroup}, for every $t_0\in [0,\varepsilon)$ and $t>0$ sufficiently small \[\frac{\mu_m(F_{t_0+t}K)-\mu_m(F_t K)}{t}=\frac{\mu_m(F_{t}F_{t_0} K)-\mu_m(F_{t_0} K)}{t}.\] Therefore, Lemma \ref{alex} shows that $\mu_m(F_t K)$ has continuous right derivative on $[0,\varepsilon]$. Since $\mu_m(K_t)$ is continuous as well, it must be differentiable in $[0,\varepsilon]$ for every $m$.
		
		Consider the decomposition $f=f^+ - f^-$, where $f^+=(f \vee 0)$ and $f^-=-(f \wedge 0)$. Then \[\int_{\mathbb{S}^{n}}f^+(\Nu)\,dS_{\mu_m,K}(\Nu) \to \int_{\mathbb{S}^{n}}f^+(\Nu)\,dS_{\mu,K}(\Nu)\] increasingly and pointwise in $t \in [0,\varepsilon]$. The same holds if we replace $f^+$ with $f^-$. 
		By Lemma \ref{lemma:weak_continuity} the two limits are continuous and bounded, and therefore \[\int_{\mathbb{S}^{n}}f(\Nu)\,dS_{\mu_m,K}(\Nu) \to \int_{\mathbb{S}^{n}}f(\Nu)\,dS_{\mu,K}(\Nu)\] uniformly in $t$ (see, for example, \cite[Theorem 7.13]{baby_rudin}).
		
		The proof is concluded since we satisfy all the hypotheses of Lemma \ref{thm:lim_der} and thus \[\lim_{m\to \infty}\lim_{t \to 0} \frac{\mu_m([h_K+tf])-\mu_m(K)}{t}=\lim_{t \to 0} \lim_{m\to \infty} \frac{\mu_m([h_K+tf])-\mu_m(K)}{t}.\]
	\end{proof}
	
	\subsection{Proof of the variational formula}
	
	We are finally ready to prove our core result. It reads as follows. 
	\begin{theorem}\label{gen_al}
		Let $u \in \Coco$ such that $\mathrm{int}(\dom(u))\neq \emptyset$, and $\zeta \in C_{rec}(\R^n)$. Consider, moreover, a measure $\mu$ on $\R^{n+1}$ such that $d\mu(z,x)=\phi(z)\psi(x)dz \,dx$ with non-negative functions $\phi \in C(\R)\cap L^1([a,+\infty))$ and $\psi \in C(\R^n)$ for some $a \in \R$ such that $\phi(z)\to 0$ as $z \to +\infty$. Then 
		\begin{align*}
			\left. \frac{d}{dt} \mu(u_t)\right|_{t=0}&=\\
			\int_{\dom(u)}\zeta(\nabla u(x))\phi(u(x))\psi(x)\, dx +\int_{\partial \dom(u)}& \rho_{\zeta}(N_{\dom(u)}(y))\Phi(u(y))\psi(x)\, d\mathcal{H}^{n-1}(y),
		\end{align*}
		exists and is finite, where $\Phi(t)=\int_t^{+\infty} \phi(z)dz$.
	\end{theorem}\noindent
	Notice that, under these assumptions, the measure of the epigraph of $u$ is finite for every $u \in \Coco$.
 \begin{remark}
     The assumption $d\mu(z,x)=\phi(z)\psi(x)dzdx$ is a simplification we adopt to lighten the notation. The interested reader can easily repeat the calculations in the hypotheses $d\mu(z,x)=\Psi(z,x)dzdx$ as in Remark \ref{rem:tr_Al}, adapting them to the proof of Theorem \ref{gen_al}.
 \end{remark}
	
	As anticipated, our strategy is to work between convex sets and convex functions. To perform this passage formally, we introduce the family 
	\begin{align*}
		&\K^{n+1}_{+}\coloneq \\&\{ K \subset \R^{n+1}: K \text{ is convex, closed, with nonempty interior, and }\text{pr}_H (K)\in \K^n_n\}.
	\end{align*}
	Notice that this family includes precisely $\K^{n+1}_{n+1}$ and the epigraphs of the functions in $\Coco$ with $n$-dimensional domain, where $\R^n\equiv H=e_{n+1}^\perp$. Let $K,L \in \K^{n+1}_{+}$ and let $\mu$ be a measure on $\R^{n+1}$. The $\mu$-symmetric-difference between $K$ and $L$ is \[d_\mu(K,L)\coloneq \mu(K \Delta L)=\int_{\R^{n+1}} \chi_{K \Delta L}(x)\,d\mu(x), \] where for a measurable set $A$ we denote by $\chi_A$ is characteristic function. In particular, if we use the Gaussian measure $\gamma_{n+1}$ on $\R^{n+1}$, defined by its density \[d\gamma_{n+1}(x)=\frac{1}{\sqrt{2\pi}^{n+1}}e^{-|x|^2/2}\,dx,\] then $d_{\gamma_{n+1}}$ defines a metric on $\K^{n+1}_{+}$. This follows from the convexity of the involved sets and the sets being of dimension $n+1$. Notice that $d_{\gamma_{n+1}}(K,L)\in [0,1]$ for every $K, L \in \K^{n+1}_{+}$. This kind of metric was studied in a wider generality by Li and Mussnig \cite{MetMuss}. Here we provide a proof that $d_{\gamma_{n+1}}$ induces an appropriate metric on $\K^{n+1}_{+}$, that is, we can approximate functions in $\Coco$ by convex bodies in $\K^{n+1}_{n+1}$. 
	\begin{lemma}\label{lemma:topology_approx}
		Let $\gamma_{n+1}$ be the Gaussian measure on $\R^{n+1}$. The function $d_{\gamma_{n+1}}:\K^{n+1}_{+}\times \K^{n+1}_{+} \to [0,1]$ is a distance. Moreover, its restrictions to $\K^{n+1}_{n+1}$ and the family of epigraphs of functions in $\Coco$ induce the topology of the Hausdorff metric and epi-convergence, respectively.
	\end{lemma}
	\begin{proof}
		The equivalence of epi-convergence and convergence with respect to $d_{\gamma_{n+1}}$ can be proved analogously to Theorem 1.2 in \cite{MetMuss}.
		
		To prove that $d_{\gamma_{n+1}}$ induces the Hausdorff metric on $\K^{n+1}_{n+1}$, we use the first part of the proof as follows: by Lemma \ref{epiconv}, the Hausdorff convergence of a sequence $\K_m \in \K^{n+1}_{n+1}$ to some $\K \in \K^{n+1}_{n+1}$ is equivalent to the epi-convergence of $I_{K_m}$ to $I_K$ as functions on $\R^{n+1}$. However, as we just proved, this is equivalent to the convergence of $\epi(I_{K_m})$ to $\epi(I_K)$ with respect to $d_{\gamma_{n+2}}$. Direct calculations show that \[d_{\gamma_{n+2}}(\epi(I_{K_m}),\epi(I_K))=Cd_{\gamma_{n+1}}(K_m,K)\] for some absolute constant $C>0$, concluding the proof.
	\end{proof}
	
	With these instruments at hand, we can start the proof of Theorem \ref{gen_al}.
	\begin{proof}[Proof of Theorem \ref{gen_al}]
		Let us first sketch the outline of the proof. We start by appropriately re-writing the variational formula \eqref{eq:variation_DylLyu} with $K=K^v+\ell_T, T>0$ for some $v \in \Coco$, $\Psi(x,z)=\phi(z)\psi(x)$, and $f=\bar{\zeta}$ as defined in \eqref{eq:zeta_proj}. The choice of $T$ is given by Lemma \ref{lemma:ex_claim}. For $\tilde{v}=\lceil K^v+\ell_T \rceil=\lceil K^v \rceil +T$, direct computations using \eqref{change1} show
		\begin{equation}\label{eq:big_bad_wolf}
			\begin{split}
				& \lim_{t \to 0} \frac{\mu([h_{K^v}+t\bar{\zeta}])-\mu(K^v )}{t}=\int_{\partial (K^v+\ell_T)} \bar{\zeta}(N_{K^v}((x,z)))\phi(z)\psi(x)\, d\haus^n((x,z))\\
				& = \int_{\dom(v)}\zeta(\nabla v(x))\phi(v(x))\psi(x)\,dx\\ & + \int_{\dom(v)} \zeta(\nabla v(x))\phi(\tilde{v}(x))\psi(x)\,dx\\ &+ \int_{\partial \dom(v)}\bar{\zeta}((N_{\dom(v)}(x),0))\left(\int_{v(x)}^{\tilde{v}(x)} \phi(s)\,ds \right) \psi(x)\,d\mathcal{H}^{n-1}(x).
			\end{split}
		\end{equation}
		In particular, notice that $\bar{\zeta}$ restricted to the equator is by definition $\rho_\zeta$. We have used the following fact: Since $v$ and $\tilde{v}$ are one a suitable reflection of the other, $\nabla v(x)=-\nabla \tilde{v}(x)$ for every $x$ where the gradient exists. For $v$, the normal to the epigraph expressed through the gradient is $(\nabla v,-1)/\sqrt{1+|\nabla v|^2}$, while for $\tilde{v}$ we have $(-\nabla \tilde{v},1)/\sqrt{1+|\nabla \tilde{v}|^2}$. Moreover, by construction $\bar{\zeta}\left( N_{K^v}(x,z) \right)=\bar{\zeta}\left( R_H(N_{K^v}(x,z)) \right)$. Since \[\bar{\zeta}\left( R_H(N_{K^v}(x,z)) \right)=\bar{\zeta}\left( \frac{-\nabla \tilde{v}(x),1}{\sqrt{1+|\nabla \tilde{v}(x)|^2}}\right)=\bar{\zeta}\left( \frac{\nabla v(x),1}{\sqrt{1+|\nabla v(x)|^2}}\right)=\zeta(\nabla v(x)), \] we have the integral on the third line of \eqref{eq:big_bad_wolf}.
		
		Recall that for $u \in \Coco, \zeta \in C_{rec}(\R^n)$ and $t>0$ sufficiently small we write $u_t \coloneqq (u^*+t\zeta)^*$. The idea is to approximate $\epi(u_t)$ with a suitable sequence of convex bodies in order to obtain our claim as a limit of the integrals of the form \eqref{eq:big_bad_wolf}. Then, we conclude by checking the hypotheses required to apply Lemma \ref{thm:lim_der}. The proof is structured in three steps after introducing the sequence approximating $\epi(u_t)$.
		
		Clearly, dim$(\epi(u))=n+1$. Consider now the segment $\ell_m=\{se_{n+1}\colon s \in [0,m]\}$ for $m \in \N$. For $T$ chosen as in Lemma \ref{lemma:ex_claim} for $u$ and $\zeta$, $K^{u}+\ell_T+\ell_m \to \epi (u)$ in the symmetric-difference distance $d_{\gamma_{n+1}}$ as $m \to \infty$. Consider the sequence of Wulff shapes $K_{m,t}\coloneq F_t(K^u+\ell_T+\ell_m)=[h_{K^u+\ell_T+\ell_m}+t\bar{\zeta}]$. Notice that $K_{m,0}$ is monotonic with respect to inclusion, and thus the sequence $\mu(K_{m,0})$ is increasing. By the properties of the sequence $K_{m,0}$ and the measure $\mu$ it follows from the monotone convergence theorem that	\[\lim_{m \to \infty} \mu(K_{m,0})=\mu(\epi(u_0))=\mu(\epi(u)).\] Moreover, notice that the sequence of sets $K_{m,t}$ converges to $\epi(u_t)$ in the symmetric-difference distance topology monotonically. In particular, Lemma \ref{lemma:ex_claim} implies that $K_{m,t}=K^{u,t}+\ell_m$, where $K^{u,t}=F_t(K^u+\ell_T)=[h_{K^u+\ell_T}+t\overline{\zeta}]$.
		
		\textit{Step 1}(Limit as $t\to 0$). Define now $u_m\coloneq \lfloor K^u+\ell_m \rfloor$ and $\tilde{u}_m\coloneq \lceil K^u+\ell_m \rceil+T$. Notice that $u_m=u$ and $\tilde{u}_m=\tilde{u}+m$, where $\tilde{u}=\lceil K^u \rceil+T$. Then, replacing $K^v$ with $K^u+\ell_m$ in \eqref{eq:big_bad_wolf}, we infer 
		\begin{equation}\label{three}
			\begin{split}
				& \lim_{t \to 0} \frac{\mu(K_{m,t})-\mu(K^u+\ell_T+\ell_m )}{t}=\int_{\partial (K^u+\ell_{T+m})} f(N_{K^u+\ell_{T+m}}((x,z)))\phi(z)\psi(x)\, d\haus^n((x,z))\\
				& = \int_{\dom(u)}\zeta(\nabla u(x))\phi(u(x))\psi(x)\,dx\\ & + \int_{\dom(u)} \zeta(\nabla u(x))\phi(\tilde{u}(x)+m)\psi(x)\,dx\\ &+ \int_{\partial \dom(u)}\rho_\zeta((N_{\dom(u)}(x))\left(\int_{u(x)}^{\tilde{u}(x)+m} \phi(s)\,ds \right) \psi(x)\,d\mathcal{H}^{n-1}(x).
			\end{split}
		\end{equation}
		
		Notice that for $m$ fixed, the integrals in the last three lines of \eqref{three} are all finite since they are just a suitable decomposition of the right-hand side of the first line. We claim that the right-hand side of \eqref{three} converges to \[\int_{\dom(u)}\zeta(\nabla u(x))\phi(u(x))\psi(x)\,dx+\int_{\partial \dom(u)} \rho_{\zeta}(N_{\dom(u)}(y))\Phi(u(y))\psi(x)\,d\mathcal{H}^{n-1}(y) \] as $m \to \infty$. The first integral is determined by the graph of $u$ and thus is fixed by the sequence we are considering. Concerning the second one, since $\tilde{u}_m=\tilde{u}+m$ is the parametrization of the upper part of $\partial(K^u+\ell_m)$, we infer $\lim_{m \to \infty} \tilde{u}_m(x)=+\infty$ for every $x \in \dom(u)$. Therefore $\lim_{m \to \infty}\phi(\tilde{u}_m(x))=0$ for every $x \in \dom(u)$ since $\phi$ converges to $0$ by hypothesis. Then, 
		\begin{align*}
			&\left|  \int_{\dom(u)} \zeta(\nabla u(x))\phi(\tilde{u}(x)+m)\psi(x)\,dx \right| \leq \\ &  \max_{x \in \dom(u)}\phi(\tilde{u}(x)+m) \int_{\dom(u)} \left| \zeta(\nabla u(x))\psi(x)\right|\,dx  .
		\end{align*}
		Since the integral on the right-hand side of the inequality is finite (it is part of the weighted surface area measure of $K^u$ when in Lemma \ref{alex} we consider $\Psi(x,z)=\psi(x)$) and \[ \max_{x \in \dom(u)}\phi(\tilde{u}(x)+m) \leq \sup_{y \in [\inf \tilde{u}+m,+\infty ]}\phi(y) \to 0\] as $m \to \infty$, the integral on the third line of \eqref{three} converges to $0$.
		
		We conclude this step observing that \[ \int_{\partial \dom(u)}\rho_{\zeta}(N_{\dom(u)}(x))\left(\int_{u(x)}^{\tilde{u}_m(x)} \phi(s)\,ds \right) \psi(x)\,d\mathcal{H}^{n-1}(x) \] converges to \[ \int_{\partial \dom(u)}\rho_{\zeta}(N_{\dom(u)}(x))\Phi(u(x)) \psi(x)\,d\mathcal{H}^{n-1}(x)\] as $m \to \infty$. Indeed $\tilde{u}_m(x) \to \infty$ as $m \to \infty$, thus \[\int_{u(x)}^{\tilde{u}_m(x)} \phi(s)\, ds \to \int_{u(x)}^{+\infty} \phi(s)\,ds=\Phi(u(x)) \] increasingly, and the limit is finite by hypothesis. Then since
		\begin{align*}
			\left| \rho_{\zeta}(N_{\dom(u)}(x))\left(\int_{u(x)}^{\tilde{u}_m(x)} \phi(s)\,ds \right) \psi(x) \right|  \leq    \left|\rho_{\zeta}(N_{\dom(u)}(x)) \psi(x) \right| \Phi(u(x)),
		\end{align*}
		the dominated convergence theorem grants the desired convergence.
		
		\textit{Step 2}(The derivative exists). We now prove that $\mu(K_{m,t})$ is differentiable with respect to $t$ for each $t$ in $[0,\varepsilon]$ for every $m \in \N$ and $\varepsilon$ suitably small, and its derivative converges uniformly as $m \to \infty$ on $[0,\varepsilon]$ to \[ \int_{\dom(u_t)}\zeta(\nabla u_t(x))\phi(u_t(x))\psi(x)\,dx+\int_{\partial \dom(u_t)}\rho_{\zeta}(N_{\dom(u)}(x))\Phi( u_t(x))\psi(x)\,d\mathcal{H}^{n-1}(x).\] For $m$ fixed, the differentiability of $\mu(K_{m,t})$ follows from Theorem \ref{semigroup}. Indeed this theorem implies \[ \frac{\mu(F_{t+t_0}(K^u+\ell_{T+m}))-\mu(F_{t_0}(K^u+\ell_{T+m}))}{t}= \frac{\mu(F_tF_{t_0}(K^u+\ell_{T+m}))-\mu(F_{t_0}(K^u+\ell_{T+m}))}{t},\] and as long as $K_{m,t_0}$ has non-empty interior we can apply Lemma \ref{alex} and have, for $t_0 \in [0,\varepsilon]$, 
		\begin{equation}\label{derivative}
			\frac{d \mu(K_{m,t})}{dt}\Big|_{t=t_0}=\int_{\mathbb{S}^{n}}\bar{\zeta}(\Nu)\,dS_{\mu,K_{m,t_0}}(\Nu). 
		\end{equation} 
		Notice that $K^u+\ell_m$ has non-empty interior for every $m$ since $K^u+\ell_1 \subset K^u+\ell_m$. Thus, we can choose $\varepsilon$ such that \eqref{derivative} remains true for every $m$ and for every $t_0 \in [0,\varepsilon]$. Since $K_{m,t}$ is continuous in $t$ by Theorem \ref{continuity} and $\,dS_{\mu,K_{m,t}}$ is weakly continuous by Lemma \ref{lemma:weak_continuity} (in the case $q=n$)	, the right derivative of $\mu(K_{m,t})$ is continuous in $t$ on $[0,\varepsilon]$. Since if the derivative of a continuous function is continuous, then the function itself is differentiable, the function $\mu(K_{m,t})$ is differentiable in $(0,\varepsilon)$ for every $m$, as desired.
		
		\textit{Step 3}(Uniform convergence). To prove the uniform convergence of the derivatives \eqref{derivative}, we start by repeating the procedure for the limits of \eqref{three} for $t \in [0,\varepsilon]$, $\varepsilon>0$ as chosen in the previous step. Consider the decomposition in \eqref{three} applied to $K_{m,t}$ for a general $t \in [0,\varepsilon]$. The first integral is independent of $m$. Indeed $\lfloor K_{m,t} \rfloor =u_t$ for every $m$ since $K_{m,t}=K^{u,t}+\ell_m$. Furthermore, $\dom(u_{m,t})=\dom(u_t)$ for every $m$. For the second integral, notice that since for $t\to 0$ one has $\lceil K_{m,t} \rceil \eqcolon \tilde{u}_{m,t}\to \tilde{u}_m=\tilde{u}+m$, we can find a sequence of values $M_m \to +\infty$ as $m \to \infty$ such that for every $t \in [0,\varepsilon]$ we have $M_m\leq \min_{\dom(u_t)}\tilde{u}_{m,t}$. Moreover, notice that for every $x \in \dom(u_t)$ we have $\nabla u_t(x)=-\nabla \tilde{u}_{m,t}(x)$ since by Lemma \ref{lemma:ex_claim} $\tilde u_{m,t}$ is the reflection of $u_t$ with respect to $H$ up to a constant. Then, for the second integral in \eqref{three}, we have
		\begin{align*}
			&\left|\int_{\dom(u_t)} \zeta(\nabla u_y(x))\phi(\tilde{u}_{m,t}(x))\psi(x)\,dx \right| \\  \leq & \left(\sup_{t \in [M_m,+\infty]}\phi(t)\right)\int_{\dom(u_t)} |\zeta(\nabla u_t(x))\psi(x)|\,dx\\
			= & \left(\sup_{t \in [M_m,+\infty]}\phi(t)\right)\int_{\dom(u_t)} |\zeta(\nabla u_t(x))\psi(x)|\,dx\\  = &  \left(\sup_{t \in [M_m,+\infty]}\phi(t)\right)\int_{{\partial K^{u,t}_{-}}}|\bar{\zeta}(g^{-1}(\nabla u_t(x)))|\psi(x)\, d\haus^{n}((x,z))\\
			\leq & \left(\sup_{t \in [M_m,+\infty]}\phi(t)\right)\max_{\Nu \in \sph^{n}} |\bar{\zeta}(\Nu)| \max_{t \in [0,\varepsilon]}\left( \int_{\partial K^{u,t}} \psi(x)\, d\haus^{n}((x,z))\right),
		\end{align*}
		where $g$ was defined in \eqref{gnomonic_projection}, and therefore the first line converges to $0$ independently of $t$. In the last inequality we have used that $\phi \geq 0$ and $\partial K^{u,t}_{-} \subset \partial K^{u,t}$. Notice that the maximum in $t$ is bounded since $K^{u,t}$ is continuous in $t$. Finally, for the last integral, the convergence is granted again by the dominated convergence theorem and making use of the fact that $\dom(u_t)=\dom(u_{m,t})$. Indeed, we get
		\begin{equation*}
			\begin{split}
				&\left|\int_{\partial \dom(u_t)}\rho_{\zeta}(N_{\dom(u_t)}(x))\Phi(u_t(x))\psi(x)\,d\mathcal{H}^{n-1}(x)- \right. \\
				& \left. \int_{\partial \dom(u_{m,t})}\rho_{\zeta}(N_{\dom(u_{m,t})}(x))\left( \int_{u_t(x)}^{\tilde{u}_{m,t}}\phi(s)\,ds \right)\psi(x)\,d\mathcal{H}^{n-1}(x) \right| \\
				=&\left| \int_{\partial \dom(u_t)}\rho_{\zeta}(N_{\dom(u_{m,t})}(x))\left( \int_{\tilde{u}_{m,t}(x)}^{+\infty} \phi(s)\,ds \right) \psi(x)\,d\mathcal{H}^{n-1}(x) \right|  \\
				\leq &\max_{\Nu \in \sph^{n-1}} |\rho_\zeta(\Nu)|\max_{t \in [0,\varepsilon]}\left|\int_{\partial \dom(u_t)} \psi(x)\,d\mathcal{H}^{n-1}(x) \right| \left( \int_{C+m/2}^{+\infty} \phi(s)\,ds \right),
			\end{split}
		\end{equation*}
		where $C= \max_{t \in [0,\varepsilon]}\max_{x \in\dom(u_t)}u_t(x)$. Notice that the maximum on $[0,\varepsilon]$ is finite since the integral is continuous in $t$, as $\dom(u_t)$ is the projection on $e_{n+1}^\perp$ of $K^{u,t}$, which is continuous in $t$. Then, as $m \to \infty$, the last integral converges to $0$ uniformly on $t$. 
		
		\textit{Conclusion}. We can now safely apply Lemma \ref{thm:lim_der} with $f_m(t)=\mu(K_{m,t})$, concluding the proof.
	\end{proof}
	
	A refinement of the same proof provides the following statement.
	\begin{theorem}\label{gen_al_sing}
		Let $u \in \Coco$ such that $o \in \mathrm{int}(\dom(u))$, and $\zeta \in C_{rec}(\R^n)$. Consider, moreover, a measure $\mu$ on $\R^{n+1}$ such that $d\mu(z,x)=\phi(z)\psi(x)|x|^{q-n}dz \,dx, q>0$, with $\phi \in C(\R)\cap L^1([a,+\infty))$ for some $a \in \R$ and $\psi \in C(\R^n)$ such that $\phi(z)\to 0$ as $z \to +\infty$. Then
		\begin{align*}
			\left. \frac{d}{dt} \mu_n(u_t)\right|_{t=0}&=\\
			\int_{\dom(u)}\zeta(\nabla u(x))\phi(u(x))\psi(x)|x|^{q-n}\, dx+\int_{\partial \dom(u)}& \rho_{\zeta}(N_{\dom(u)}(y))\Phi(u(y))\psi(x)|y|^{q-n}\, d\mathcal{H}^{n-1}(y),,
		\end{align*}
		exists and is finite, where $\Phi(t)=\int_t^{+\infty} \phi(z)dz$.
	\end{theorem}\noindent
	\begin{proof}
		Consider the sets $K_{m,t}$ obtained through the same construction as in the proof of Theorem \ref{gen_al}. Since $o \in \mathrm{int}(\dom(u))$, for $\varepsilon>0$ suitably small $o \in \mathrm{pr_H}(\mathrm{int}K_{m,t})$ for every $m$. Then, we can repeat the proof verbatim since Lemma \ref{lemma:weak_continuity} and Corollary \ref{cor:alex} can be applied again.
	\end{proof}
	
	\section{Application to moment measures}
	
	Consider the case $\phi(z)=e^{-z}$, $\psi \equiv 1$. This point of view has been initially investigated by Colesanti and Fragalà \cite{ColFrag} and Cordero-Erausquin and Klartag \cite{CorKlar}. The two interpretations stem from independent perspectives. The former aims to generalize some classical concepts from the Brunn-Minkowski theory. The latter originates from the interest in the KLS conjecture and solutions of differential equations in Kähler-Einstein manifolds; see, for example, Klartag \cite{KlaMM}. The first approach required $C^2$ regularity on the interior of the domain, while the second was restricted to essentially continuous convex functions (see \cite[Definition 2]{CorKlar}). Recently, Rotem \cite{RotSur1,RotSur2} improved significantly these results, dropping all regularity assumptions with the result in Theorem \ref{thm:intro_rot}. Theorem \ref{gen_al} immediately implies the following variant, which we reported in the introduction as Theorem \ref{thm:intro_gen}
	\begin{corollary}\label{cor:Rotem}
		Consider $u \in \Coco$, $\zeta \in C_{rec}(\R^n)$, and the measure $\mu$ in $\R^{n+1}$ such that $d\mu(x,z)=e^{-z}dzdx$. Then, 
		\begin{equation*}
			\left. \frac{d}{dt}\mu(u_t)\right|_{t=0}=\int_{\dom(u)}\zeta(\nabla u(x))e^{-u(x)}\,dx+\int_{\partial \dom(u)}\rho_\zeta(N_{\dom(u)}(y))e^{-u(y)}\,d\mathcal{H}^{n-1}(y).
		\end{equation*}
	\end{corollary}
	
	We can then move on to the measures studied in \cite{dual_curvature_log}, which, in our setting, correspond to the choice $d\mu(x,z)=d\mu_q(x,z)=e^{-z}|x|^{q-n}dxdz$. The functional volume then is \[\mu_q(u)\coloneqq \int_{\dom(u)} e^{-u(x)}|x|^{q-n}\, dx\] for $u \in \Coco$ such that $o \in \mathrm{int}(\dom(u))$. As proved in \cite[Proposition 1.2]{dual_curvature_log}, if $u$ is coercive, this integral is finite, which is always the case if $u \in \Coco$. 
	Moreover, Theorem 1.3 in the same work provides the following variational result. 
	\begin{theorem}[Huang, Liu, Xi, and Zhao]
		\label{thm:dual_Zhao}
		Consider $u\in \Cosc(\R^n)$ such that $o \in \mathrm{int}(\dom(u))$, and choose $q>0$. Assume $u$ achieves its minimum at the origin and
		\begin{equation}\label{eq:hp_dual}
			\limsup_{|x|\to 0} \frac{|e^{-u(x)}-e^{-u(0)}|}{|x|^{\alpha+1}}<\infty,
		\end{equation}
		for some $0<\alpha<1$.
		
		Consider, moreover, $v \in \Coco$ such that $o \in \dom(v)$. Then,
		\begin{align*}
			&\lim_{t\to 0^+} \frac{\mu_q(\epi(u\square(t\sq v)))-\mu_q(\epi(u))}{t} =\\ \int_{\R^n} v^*(\nabla u(x)) &e^{-u(x)}|x|^{q-n}\, dx +\int_{\partial \dom(u)} h_{\dom(v)}(N_{\dom(u)}(y)) e^{-u(y)}|y|^{q-n}\,d\haus^{n-1}(y).
		\end{align*}
	\end{theorem}
	
	A direct application of Theorem \ref{gen_al_sing} shows that the hypothesis \eqref{eq:hp_dual} is not always necessary for $u \in \Coco$. Moreover, we can consider all deformations in $C_{rec}(\R^n)$. Compare with \cite[Lemma 5.2]{dual_curvature_log}.
	\begin{corollary}\label{cor:dual_Rotem}
		Consider $u \in \Coco$ and $\zeta \in C_{rec}(\R^n)$, and $\mu$ the measure in $\R^{n+1}$ such that $d\mu(x,z)=e^{-z}|x|^{q-n}dzdx$. Suppose, moreover that $o \in \mathrm{int}(\dom(u))$. Then
		\begin{align}
			\label{eq:dual_log} \left. \frac{d}{dt}\mu_q(u_t)\right|_{t=0}&=\\\int_{\R^n} & \zeta(\nabla u(x))e^{-u(x)}|x|^{q-n}\, dx + \int_{\partial \dom(u)}\rho_\zeta(N_{\dom(u)}(y))e^{-u(y)}|y|^{q-n}\,d\mathcal{H}^{n-1}(y).  \notag
		\end{align}
	\end{corollary}
	
	The two integrals in \eqref{eq:dual_log} are finite whenever $u \in \Cosc(\R^n)$ and $\zeta$ is Lipschitz (thus including the case $\zeta=v^*,v \in \Coco$), as the following lemma shows. The finiteness of the first one is a consequence of \cite[Theorem 5.12]{dual_curvature_log}, while for the second one, we adapt the idea from \cite[Proposition 1.6]{RotSur2}, which in turn is inspired by an argument by Ball \cite{BallRev}.
	\begin{lemma}\label{lemma:dual_finiteness}
		Consider $u \in \Cosc(\R^n)$ such that $o \in \mathrm{int}(\dom(u))$. If $\zeta \in C_{rec}(\R^n)$ is Lipschitz, the integrals
		\begin{align*}
			\int_{\R^n} \left| \zeta(\nabla u(x)) \right| e^{-u(x)}|x|^{q-n}\, dx \text{ and }\\ \int_{\partial \dom(u)} \left| \rho_\zeta(N_{\dom(u)}(y))\right| e^{-u(y)}|y|^{q-n}\,d\haus^{n-1}(y)
		\end{align*}
		are finite.
	\end{lemma}
	\begin{proof}
		First notice that since $\zeta$ is a Lipschitz function there exists $L \geq 0$ such that for every $x \in \R^n$ \[|\zeta(x)|-|\zeta(0)|\leq |\zeta|-|\zeta(0)||\leq L|x|. \] That is, there exists $a\geq 0$ such that \[|\zeta(x)|\leq L|x|+a\] for every $x \in \R^n$. Then \[  \int_{\R^n} \left|\zeta(\nabla u(x))\right|e^{-u(x)}|x|^{q-n}\, dx  \leq \int_{\R^n} L|\nabla u(x)|e^{-u(x)}|x|^{q-n}\, dx+\int_{\R^n} ae^{-u(x)}|x|^{q-n}\, dx. \] The first integral on the right is finite thanks to \cite[Theorem 5.12]{dual_curvature_log}, while the second is finite by \cite[Proposition 1.2]{dual_curvature_log} since $u$ is coercive.
		
		Consider now the second integral in the statement. Since $u$ is convex and coercive there exist $A,c>0$ such that $e^{-u(x)}\leq Ae^{-c|x|}$ (see \cite[Lemma 2.1]{KlarUnif}). Moreover, notice that \[e^{-c|x|}=c\int_0^{+\infty} e^{-ct} \mathds{1}_{tB^n}(x) dt \] for every $x \in \R^n$, where $\mathds{1}_{tB^n}=e^{-I_{tB^n}}$ is the characteristic function of the Euclidean unit ball. Then we infer 
		\begin{equation}\label{eq:estimate_boundary}
			\begin{split}
				\int_{\partial \dom(u)} e^{-u(y)}&|y|^{q-n}\, d\haus^{n-1}(y)\leq A \int_{\partial \dom(u)}e^{-c|y|}|y|^{q-n}\, d\haus^{n-1}(y) \\= Ac &\int_{\partial \dom(u)}\int_0^{+\infty} e^{-ct} \mathds{1}_{tB^n}(y) |y|^{q-n}\, dtd\haus^{n-1}(y)\\ = Ac &\int_0^{+\infty} e^{-ct}\int_{\partial \dom(u) \cap tB^n}|y|^{q-n}\, d\haus^{n-1}(y)dt.
			\end{split}
		\end{equation}
		If $0<q<n$ we continue \eqref{eq:estimate_boundary} by 
		\begin{align*}
			\ldots \leq Ac\left(\int_0^1e^{-ct}\int_{\partial tB^n}|y|^{q-n}\, d\haus^{n-1}(y)\, dt+\int_1^{+\infty}e^{-ct}\int_{\partial tB^n}\, d\haus^{n-1}(y)\, dt\right) \\
			= Ac\haus^{n-1}(\sph^{n-1})\left( \int_0^{1} e^{-ct}t^{q-1}\, dt + \int_1^{+\infty} e^{-ct}t^{n-1}\, dt \right).
		\end{align*}
		If, alternatively, $n\leq q$, analogously we obtain
		\begin{align*}
			\ldots \leq Ac\left(\int_0^1e^{-ct}\int_{\partial tB^n}\, d\haus^{n-1}(y)\, dt+\int_1^{+\infty}e^{-ct}\int_{\partial tB^n}\, d\haus^{n-1}(y)|y|^{q-n}\, dt\right) \\
			= Ac\haus^{n-1}(\sph^{n-1})\left( \int_0^{1} e^{-ct}t^{n-1}\, dt + \int_1^{+\infty} e^{-ct}t^{q-1}\, dt \right).
		\end{align*}
		In any case, since the integrals in both the last lines are finite, all the preceding ones are. Observing that \[ \int_{\partial \dom(u)}\left| \rho_\zeta(N_{\dom(u)}(y))\right| e^{-u(y)}|y|^{q-n}\,d\haus^{n-1}(y) \leq L\int_{\partial \dom(u)} e^{-u(y)}|y|^{q-n}\, d\haus^{n-1}(y), \]the proof is concluded.
	\end{proof}\noindent
	We now report two estimates from the proof of \cite[Theorem 5.12]{dual_curvature_log}. We use again the fact that if $u \in \Conv(\R^n)$ is coercive then there exist $A,c>0$ such that $e^{-u(x)}\leq Ae^{c|x|}$. In this lemma we write $(x_1,\dots,x_n)=x \in \R^n$ as $x=(\bar{x},x_n), \bar{x}\in \R^{n-1}$.
	\begin{lemma}\label{lemma:estimates}
		Consider $u \in \Cosc(\R^n)$ and let $A,c>0$ be such that $e^{-u(x)}\leq Ae^{-c|x|}$ for every $x \in \R^n$.  Then, for every $m \in \N$:
		\begin{itemize}
			\item If $0<q\leq n$
			\begin{equation}\label{eq:estimate_1}
				\begin{split}
					&\int_{\R^n\setminus mB^n} \left| \frac{\partial}{\partial x_n} e^{-u(x)}\right||x|^{q-n}\, dx\leq \\
					4A&\int_{\left( \R^n \setminus \frac{m}{2}B^n\right) \cap \R^{n-1}} \sup_{x_n\in \R} e^{-c|(\bar{x},x_n)|}|\bar{x}|^{q-n}\, d\bar{x}\\+4A&\left(\frac{\sqrt{3}}{2}m \right)^{q-n}\int_{\frac{m}{2}B^n \cap \R^{n-1}} \sup_{x_n\in\R} e^{-c|(\bar{x},x_n)|}\, d\bar{x},
				\end{split}
			\end{equation}
			\item If $q>n$ there exists a constant $C$ (independent of $u$)
			\begin{equation}\label{eq:estimate_2}
			\begin{split}
					&\int_{\R^n\setminus mB^n} \left| \frac{\partial}{\partial x_n} e^{-u(x)}\right||x|^{q-n}\, dx\leq \\ 4CA \sum_{k=m}^{\infty}&(k+1)^{q-n}\left((k+1)^{n-1}e^{-c(k-1)}+(k-1)^{n-1}e^{c\sqrt{k}} \right).
			\end{split}
			\end{equation}
		\end{itemize}
	\end{lemma}\noindent
	Furthermore, we need the following approximation result.
	\begin{lemma}\label{lemma:good_approx}
		Consider $u \in \Conv(\R^n)$, $v \in \Coco$ and $t\geq 0$. Then, for every $m > \max_{y \in \dom(v)} |y|$ there exists $\delta(m)>0$ such that $\lim_{m\to \infty} \delta(m)=+\infty$ and \[u\square(t\sq v)(x)=(u+I_{m' B^n})\square(t\sq v)(x)\] for every $|x| \leq \delta(m),0\leq t \leq 1$, and $m' \geq m$.
	\end{lemma}
	\begin{proof}
		We choose $\delta(m)\coloneq m-\max_{y \in \dom(v)} |y|$. Notice that, by construction, \[u\square(t\sq v)(x)\leq (u+I_{m B^n})\square(t\sq v)(x)\] for every $x \in \R^n$ and every $m$. 
		
		Consider $ m > \max_{y \in \dom(v)} |y|$ and $|x|\leq \delta(m)$. If $u\square(t\sq v)(x)=+\infty$, there is nothing to prove. If $u\square(t\sq v)(x)< +\infty$, then necessarily $x \subset \dom(v)+mB^n$. Therefore, \[u\square(t\sq v)(x)=\inf_{y \in mB^n}\{u(y)+t\sq v(x-y)\}=\inf_{y \in \R^n}\{u(y)+I_{ mB^n}(y)+t\sq v(x-y)\}=(u+I_{m' B^n})\square(t\sq v)(x)\] for every $0\leq t \leq 1$ and $m' \geq m$.
	\end{proof}
	
	A final approximation procedure, going through Corollary \ref{cor:dual_Rotem}, proves Theorem \ref{thm:dual_new_intro}, lifting the assumption \eqref{eq:hp_dual} in Theorem \ref{thm:dual_Zhao}. We rewrite the statement for the convenience of the reader. 
	\begin{theorem*}
		Consider $u\in \Cosc(\R^n)$ such that $o \in \mathrm{int}(\dom(u))$, and choose $q>0$. Consider $v \in \Coco$ such that $o \in \dom(v)$. Then
		\begin{align*}
			&\left. \frac{d}{dt} \mu_q(\epi(u\square(t\sq v))) \right|_{t=0^+} =\\ \int_{\R^n} v^*(\nabla u(x)) e^{-u(x)}|x|^{q-n}\, &dx +\int_{\partial \dom(u)} h_{\dom(v)}(N_{\dom(u)}(y)) e^{-u(y)}|y|^{q-n}\,d\haus^{n-1}(y).
		\end{align*}
	\end{theorem*}
	\begin{proof}[Proof of Theorem \ref{thm:dual_new_intro}] Again, our strategy is to use Lemma \ref{thm:lim_der} to prove our claim by approximation. Consider $u$ as in the hypotheses. We approximate $u$ through the sequence $u_m\coloneqq u+I_{mB^n}$. By \cite[Proposition 4.2]{Att_Wets} $u_m\square (t\sq v)$ epi-converges to $u\square(t\sq v)$ for every $t\geq 0$. Moreover, $u_m\square (t\sq v) \leq u\square(t\sq v)$ for every $m$. Therefore, for $\varepsilon>0$ suitably small we can find $A,c >0$ such that
		\begin{equation}\label{eq:unif_bound}
			e^{u_m\square (t\sq v)}\leq e^{-u\square(t\sq v)(x)}\leq Ae^{-c|x|}
		\end{equation}
		for every $x \in \R^n, m \in \N$, and $t \in [0,\varepsilon]$. Without loss of generality, we suppose $o \in \epi(v)$ and $\max_{y \in \dom(v)} v(y)=0$, so that \[h_{K^{u_m}}+t\overline{v^*}=h_{K^{u_m}}+th_{K^v}.\] This choice affects the calculation with a multiplicative constant that we can ignore by Remark \ref{rem:tr_Al}. For the sake of conciseness, denote $u_t=u\square (t\sq v)$ and $u_{m,t}=u_m\square (t\sq v)$.
		
		By \eqref{eq:unif_bound}
		\begin{align*}
			\int_{\R^n}e^{-u(x)}|x|^{q-n}\, dx - \int_{\R^n}e^{-u_m(x)}|x|^{q-n}\, dx=\\ \int_{\R^n\setminus mB^n}e^{-u(x)}|x|^{q-n}\, dx \leq \int_{\R^n\setminus mB^n}Ae^{-c|x|}|x|^{q-n}\, dx\to 0
		\end{align*}
		as $m \to \infty$. Thus $\mu(\epi(u_m))\to \mu(\epi(u))$.
		
		We now want to prove that the derivative in $t$ of $\mu(\epi(u_{m,t}))$ converges uniformly to \[\int_{\R^n} v^*(\nabla u(x)) e^{-u(x)}|x|^{q-n}\, dx +\int_{\partial \dom(u)} h_{\dom(v)}(N_{\dom(u)}(y)) e^{-u(y)}|y|^{q-n}\, d\haus^{n-1}(y)\] for $t \in [0,\varepsilon]$, where $\varepsilon>0$ is small enough for Corollary \ref{cor:dual_Rotem} to hold as well. We do so approximating the two integrals separately by the ones in \eqref{eq:dual_log} evaluated for $u_m$ with the perturbation $v \in \Coco$ (that is, $v^* \in C_{rec}(\R^n)$).
		
		We start with the integral on $\R^n$. Notice that since $o \in \epi(v)$, we have $v^*\geq 0$. By Lemma \ref{lemma:good_approx}, and since $\dom(u_{m,t})\subset mB^n+\dom(v)$ for $m$ suitably large,
		\begin{align*}
			\left| \int_{\R^n} v^*(\nabla u_t(x)) e^{-u_t(x)}|x|^{q-n}\, dx - \int_{\R^n} v^*(\nabla u_{m,t}(x)) e^{-u_{m,t}(x)}|x|^{q-n}\, dx\right|\\
			\leq 	\left|\int_{\R^n} v^*(\nabla u_t(x) ) e^{-u_t(x)}|x|^{q-n}\, dx- \int_{ \delta(m)B^n} v^*(\nabla u_{m,t}(x)) e^{-u_{m,t}(x)}|x|^{q-n}\, dx \right|	\\ +\left|  \int_{ (mB^n+\dom(v))\setminus \delta(m)B^n} v^*(\nabla u_{m,t}(x)) e^{-u_{m,t}(x)}|x|^{q-n}\, dx\right|\\\leq \int_{\R^n\setminus \delta(m)B^n} v^*(\nabla u_t(x) ) e^{-u_t(x)}|x|^{q-n}\, dx + \int_{\R^n\setminus \delta(m)B^n} v^*(\nabla u_{m,t}(x) ) e^{-u_{m,t}(x)}|x|^{q-n}\, dx
		\end{align*}
		By Lemma \ref{lemma:estimates}, this last line converges to $0$ independently of $t$, since by \eqref{eq:unif_bound} we can use \eqref{eq:estimate_1} or \eqref{eq:estimate_2} depending on $q$ for both the integrals there.
		
		We now repeat a similar procedure for the integral on the boundary of $\partial \dom(u)$. If $u \in \Coco$, we have nothing to prove since for $m$ sufficiently large, one obtains exactly the desired quantity. If $\dom(u)=\R^n$ (and thus $\dom(u_t)=\R^n$) then \[\int_{\partial \dom(u_{m,t})} h_{\dom(v)}(N_{\dom(u_{m,t})}(y)) e^{-u(y)}|y|^{q-n}\, d\haus^{n-1}(y)=0\] converges uniformly to $0$ by the estimates following \eqref{eq:estimate_boundary}. If $\dom(u)\neq \R^n$, then for $m$ sufficiently large $\partial \dom(u_t)\cap \partial \dom(u_{m,t}) \neq \emptyset$. Notice, moreover, that $\partial \dom(u_{m,t})\setminus \partial \dom (u_{t}) \subset \R^n \setminus mB^n$ and $\partial \dom(u_t)\setminus \partial \dom (u_{m,t}) \subset \R^n \setminus mB^n$ for every $m$. We infer, following the same strategy in \eqref{eq:estimate_boundary} and denoting by $L$ the Lipschitz constant of $v^*$,
		\begin{align*}
			\left|	\int_{\partial \dom(u_t)}h_{\dom(v)}(N_{\dom(u_t)}(y))e^{-u_t(y)}|y|^{q-n}\, d\haus^{n-1}(y) \right. -\\
			\left. \int_{\partial \dom(u_{m,t})}h_{\dom(v)}(N_{\dom(u_{m,t})}(y))e^{-u_{m,t}(y)}|y|^{q-n}\, d\haus^{n-1}(y) \right| \\ 
			\leq \int_{\partial \dom(u_t)\setminus \partial \dom (u_{m,t})}h_{\dom(v)}(N_{\dom(u_t)}(y))e^{-u_t(y)}|y|^{q-n}\, d\haus^{n-1}(y)\\ 
			+ \int_{\partial \dom(u_{m,t})\setminus \partial \dom (u_{t})}h_{\dom(v)}(N_{\dom(u_{m,t})}(y))e^{-u_{m,t}(y)}|y|^{q-n}\, d\haus^{n-1}(y)\\
			\leq 2L Ac \int_m^{\infty} e^{-c\tau}\int_{\partial \dom(u) \cap \tau B^n}|y|^{q-n}\, d\haus^{n-1}(y)d\tau.
		\end{align*}
		If $0<q<n$ we continue by \[\ldots \leq 2L Ac \int_m^{\infty} e^{-c\tau}\tau^{n-1}d\tau.\] Otherwise, if $n\leq q$, we conclude by \[\ldots \leq  2L Ac \int_m^{\infty} e^{-c\tau}\tau^{q-1}d\tau. \] In both cases, we have uniform convergence to $0$, concluding the proof. 
	\end{proof}
	
	\begin{remark}
		We conclude this section noting that the same method in the proof of Theorem \ref{thm:dual_new_intro} can be specialized to the case $q\geq n$ without requiring any assumption on the particular position of the epigraphs of $u$ and $v$. In particular, it is possible to recover Theorem \ref{thm:intro_rot} when $v \in \Coco$.
	\end{remark}

\subsection*{Acknowledgments} The core of this work originated from the Ph.D. thesis of the author, supervised by Gabriele Bianchi and Paolo Gronchi, whose support and guidance were fundamental. We extend our gratitude to Monika Ludwig, Fabian Mussnig, Andrea Colesanti, and Donmeng Xi for the many fruitful discussions.

\subsection*{Funding}
The author was supported, in part, by the Austrian Science Fund, (FWF): 10.55776/P34446.


\bigskip

\parbox[t]{8.5cm}{
Jacopo Ulivelli\\
Institut f\"ur Diskrete Mathematik und Geometrie\\
TU Wien\\
Wiedner Hauptstra{\ss}e 8-10/1046\\
1040 Wien, Austria\\
e-mail: jacopo.ulivelli@tuwien.ac.at}
\end{document}